\documentclass[11pt,reqno]{amsart}
\usepackage{amsmath,amsxtra,latexsym,amsthm,amssymb,amscd,pb-diagram}
\usepackage{mathrsfs,mathabx,amsfonts}
\usepackage[colorlinks=true,linkcolor=red,citecolor=blue]{hyperref}
\usepackage[margin=1in]{geometry}
\usepackage{bm}
\usepackage{color}
\usepackage{makecell,multirow,diagbox}
\usepackage{mathtools}
\usepackage[mathlines]{lineno}
\usepackage{verbatim} 
\usepackage{graphicx}
\usepackage{epsfig}
\usepackage{array}
\usepackage{enumitem}
\usepackage{silence}
\newtheorem{definition}{Definition}[section]
\newtheorem{theorem}{Theorem}[section]
\newtheorem{proposition}{Proposition}[section]
\newtheorem{lemma}{Lemma}[section]

\newtheorem{remark}{Remark}[section]

\begin{document} 
\title[Damped wave equation with Hartree-type nonlinearity]{A Fujita-type threshold for the semilinear damped wave equation with Hartree-type nonlinearity and initial data from homogeneous Besov spaces}

\thanks{$^*$\textit{Corresponding author:} Duc An Phan}

\maketitle
\centerline{\textbf{Thuy Hien Thi Nguyen}}
{\footnotesize
	\centerline{Faculty of Statistics, National Economics University}
	\centerline{207, Giai Phong Street, Dong Tam, Hai Ba Trung, Hanoi, Vietnam}
	\centerline{Email: nth.nguyenthuyhien@gmail.com}}
\medskip

\centerline{\textbf{Duc An Phan}}
{\footnotesize
	\centerline{Department of Mathematics, Banking Academy of Vietnam}
	\centerline{12 Chua Boc, Kim Lien, Hanoi, Vietnam}
	\centerline{Email: anpd@hvnh.edu.vn}}

\begin{abstract}
In this paper, we study the semilinear damped wave equation with Hartree-type nonlinearity $\mathcal{I}_\gamma(|u|^{p_1})|u|^{p_2}$, where $0<\gamma<n$, and initial data possessing additional negative regularity in $\dot{B}_{2,\infty}^{-\beta}$, $0\leq\beta<n/2$. We first establish decay estimates for the corresponding linear equation in homogeneous Besov spaces $\dot{B}_{q,r}^s$, allowing low-frequency data in $\dot{B}_{m,\infty}^{-\beta}$ under the condition $s+\beta+n(1/m-1/q)>0$. We then prove the existence of a unique global mild solution for sufficiently small initial data in the critical and supercritical range
$$
p_1+p_2\geq1+\frac{4+2\gamma}{n+2\beta},
$$
under the remaining assumptions on $p_1$ and $p_2$. Furthermore, for initial data satisfying a positive lower bound, we prove the nonexistence of global weak solutions when
$$
2<p_1+p_2<1+\frac{4+2\gamma}{n+2\beta}.
$$
\end{abstract}

\tableofcontents

\section{Introduction}
\subsection{Background}
We begin with the semilinear damped wave equation with power-type nonlinearity
\begin{equation} \label{eq:Matsumura1976}
\begin{cases}
\partial_t^2 u -\Delta u+\partial_t u=|u|^p, & (t, x) \in(0, \infty) \times \mathbb{R}^n, \quad n \in \mathbb{N}^*, \quad p>1, \\ u(0, x)=u_0(x), \quad \partial_t u(0, x)=u_1(x), & x \in \mathbb{R}^n, \quad n \in \mathbb{N}^*.
\end{cases}
\end{equation}
For initial data with additional $L^1$ integrability, global existence and finite-time blow-up for \eqref{eq:Matsumura1976} have been studied extensively in \cite{Ikehata2005, Matsumura1976, Todorova2001, Zhang2001} and related works. The critical behavior with respect to the exponent $p$ is determined by the Fujita exponent $p_{\mathrm{Fuji}}(n) := 1 + 2/n$; see \cite{Ikehata2005, Matsumura1976, Todorova2001, Zhang2001} for details. More precisely:
\begin{itemize}
\item In dimensions $n = 1, 2$, Matsumura \cite{Matsumura1976} proved the global existence of small-data solutions for $p > p_{\mathrm{Fuji}}(n)$.
\item For all $n \geq 1$, Todorova and Yordanov \cite{Todorova2001} established global existence for compactly supported small initial data when $p > p_{\mathrm{Fuji}}(n)$, with the additional restriction $p\leq n/(n-2)$ for $n\geq3$, and proved finite-time blow-up in the subcritical range $1 < p < p_{\mathrm{Fuji}}(n)$.
\item At the critical exponent $p = p_{\mathrm{Fuji}}(n)$, Zhang \cite{Zhang2001} proved blow-up.
\end{itemize}

Subsequent studies considered \eqref{eq:Matsumura1976} under different assumptions on the initial data. In particular, for data with additional $L^m$ integrability, where $m \in [1,2)$, the corresponding modified Fujita exponent is $p_{\mathrm{Fuji}}(n/m):=1+(2m)/n$.

For slowly decaying data with additional $L^m$ integrability, $m \in (1,2)$, the quantity $p_{\mathrm{Fuji}}(n/m)=1+(2m)/n$ appears as the relevant threshold; see, for instance, \cite{Ikehata2002, Ikeda2019}. In particular, the critical power belongs to the small-data global existence range for $m>1$ in the setting of \cite{Ikeda2019}, whereas at the endpoint $m=1$ the classical Fujita critical power lies in the blow-up range. This reflects a basic distinction between the $L^1$ theory and the $L^m$, $m>1$, theory. For data with additional negative Sobolev regularity $\dot{H}^{-\beta}$ with $\beta \in (0, n/2)$, Chen and Reissig \cite{Chen2023}, under suitable restrictions on the dimension and the negative Sobolev order, identified the threshold
$$
p_{\mathrm{crit}}(n,\beta):=1+\frac{4}{n+2\beta}.
$$
They obtained global (in time) existence of small-data solutions in the corresponding supercritical range and finite-time blow-up for $1<p<p_{\mathrm{crit}}(n,\beta)$. The critical case $p=p_{\mathrm{crit}}(n,\beta)$ was subsequently treated by D'Abbicco \cite{DAbbicco2025} under suitable restrictions on the dimension and the negative Sobolev order.

More recently, Tang, Duong and Phan \cite{Loc2026} considered the semilinear damped wave equation with Riesz potential-type power nonlinearity and initial data in pseudo-measure spaces:
\begin{align}\label{Semilinear_Damped_Waves}
\begin{cases}
\partial_t^2 u-\Delta u+\partial_t u=\mathcal{I}_\gamma\left(|u|^p\right), &(t, x) \in(0, \infty) \times \mathbb{R}^n, \quad n \in \mathbb{N}^*,\\
u(0,x)=u_0(x), \quad \partial_t u(0,x)=u_1(x),& x \in \mathbb{R}^n, \quad n \in \mathbb{N}^*,
\end{cases}
\end{align}
where $\gamma \in (0, n)$ and $p>1$. The Riesz potential $\mathcal{I}_\gamma$ is given by
\begin{equation} \label{eq:Riesz}
\mathcal{I}_\gamma (f)(x):= \frac{\Gamma\left(\frac{n-\gamma}{2}\right)}{2^\gamma \pi^{\frac{n}{2}} \Gamma\left(\frac{\gamma}{2}\right)}\left(|x|^{-(n-\gamma)} * f\right)= \frac{\Gamma\left(\frac{n-\gamma}{2}\right)}{2^\gamma \pi^{\frac{n}{2}} \Gamma\left(\frac{\gamma}{2}\right)} \int_{\mathbb{R}^n} \displaystyle\frac{f(y)}{|x-y|^{n-\gamma}} d y,
\end{equation}
initially for $f\in\mathcal{S}\left(\mathbb{R}^n\right)$, and then extended to suitable function spaces. Equivalently, $\mathcal{I}_\gamma$ can be viewed as the inverse fractional Laplacian in the sense that
$$
\mathcal{I}_\gamma (f)(x) = (-\Delta)^{-\frac{\gamma}{2}} f(x)=\mathcal{F}^{-1}(|\xi|^{-\gamma} \widehat{f}(\xi))(x).
$$
Throughout, we use the Fourier transform convention $\widehat{f}(\xi):=\int_{\mathbb{R}^n}\mathrm{e}^{-ix\cdot{}\xi}f(x)dx$ and $\mathcal{F}^{-1}(g)(x):=(2\pi)^{-n}\int_{\mathbb{R}^n}\mathrm{e}^{ix\cdot{}\xi}g(\xi)d\xi$.
We refer to \cite{Landkof1972, Stein1970} for standard properties of Riesz potentials. The pseudo-measure spaces are defined by (see, for instance, \cite{Bhattacharya2003})
$$
\mathcal{Y}^q\left(\mathbb{R}^n\right):=\left\{f \in \mathcal{S}'\left(\mathbb{R}^n\right): \widehat{f} \in L_\mathrm{loc}^1\left(\mathbb{R}^n\right) \text{ and } \|f\|_{\mathcal{Y}^q} :=  \sup _{\xi \in \mathbb{R}^n}\left\{|\xi|^q|\widehat{f}(\xi)|\right\}<\infty\right\}, 
$$
where $\mathcal{S}'\left(\mathbb{R}^n\right)$ denotes the space of tempered distributions. Assuming in addition that the initial data belong to $\mathcal{Y}^q$, the authors obtained for \eqref{Semilinear_Damped_Waves} the threshold $p_{\mathrm{Fuji}}((2n-2q)/(2+\gamma)):=1+(2+\gamma)/(n-q)$ with $1 \leq n \leq 4$ and $0 \leq \gamma < q < n/2$.

Hartree-type convolution nonlinearities also arise in parabolic equations with fractional diffusion. A representative model is
\begin{equation}\label{ref_eq}
\begin{cases}
\partial_t u+(-\Delta)^{\sigma/2} u=\left(\mathcal{K}*|u|^{p_1}\right)|u|^{p_2}, &(t, x) \in(0, \infty) \times \mathbb{R}^n, \quad n \in \mathbb{N}^*, \\
u(0, x)=u_0(x), &x \in \mathbb{R}^n, \quad n \in \mathbb{N}^*,
\end{cases}
\end{equation}
where $p_1, p_2 >0$ and $\sigma\in(0,2]$. Here $(-\Delta)^{\sigma/2}$ denotes the fractional Laplacian, while the nonlinear term contains the convolution kernel $\mathcal{K}$. The
function $\mathcal{K}:(0,\infty)\to(0,\infty)$ is assumed to be continuous and locally integrable in the sense that $\mathcal{K}(|\cdot{}|)\in L_\mathrm{loc}^1$, together with a suitable structural condition at infinity. The convolution term is defined by
$$
\left(\mathcal{K} *|u|^{p_1}\right)(x)=\int_{\mathbb{R}^n} \mathcal{K}(|x-y|)|u(y)|^{p_1} d y,
$$
which describes a nonlocal interaction in space. Typical choices of $\mathcal{K}$ include constant kernels and kernels of the form $\mathcal{K}(r)=r^{-\eta}$ with $\eta \in(0, n)$, or more generally, $\mathcal{K}(r)=r^{-\eta} \ln^\kappa(1+r)$ with $\eta \in(0, n)$ and $\kappa>\eta-n$. A particularly relevant choice is
$$
\mathcal{K}(r)=\frac{\Gamma\left(\frac{n-\gamma}{2}\right)}{2^\gamma \pi^{\frac{n}{2}} \Gamma\left(\frac{\gamma}{2}\right)} r^{-(n-\gamma)}, \quad \text{ with } \gamma \in(0, n)
$$
for which the convolution operator is exactly the Riesz potential
$$
\mathcal{I}_\gamma(f)(x)=\frac{\Gamma\left(\frac{n-\gamma}{2}\right)}{2^\gamma \pi^{\frac{n}{2}} \Gamma\left(\frac{\gamma}{2}\right)} \int_{\mathbb{R}^n} \frac{f(y)}{|x-y|^{n-\gamma}} d y.
$$
Fino and Torebek \cite{Fino_Hartree} studied the global behavior of solutions to the semilinear parabolic problem \eqref{ref_eq} with fractional diffusion and nonlocal convolution nonlinearities, with particular attention to global existence and blow-up. In their existence theory for the Riesz-potential case, they assume $p_1>1$ and $p_2\geq1$, whereas some of their nonexistence results allow positive exponents. The stationary cases of \eqref{ref_eq} with $\mathcal{K}(r)=r^{\gamma-n}, \gamma \in(0,n)$, including questions concerning critical exponents and existence of solutions, have been investigated in \cite{Ghergu2, Ghergu1, Moroz1, Moroz2}.

Convolution-type nonlinearities in evolution equations are also connected with mean-field models in quantum mechanics. The term "Hartree-type" comes from Hartree's self-consistent-field approximation for atomic systems; see \cite{Hartee1, Hartee2, Hartee3}. A standard time-dependent Hartree equation is
$$
\begin{cases}
i\partial_t u-\Delta u=\left(|x|^{\gamma-n}* |u|^2\right)u, &(t, x) \in(0, \infty) \times \mathbb{R}^n, \quad n \in \mathbb{N}^*,\\
u(0,x)=u_0(x), &x \in \mathbb{R}^n, \quad n \in \mathbb{N}^*.
\end{cases}
$$
For $n=3$ and $\gamma=2$, the associated stationary problem is related to the Choquard equation, which also appears in Pekar's polaron model \cite{Pekar}.

Nonexistence results for parabolic equations with Hartree-type
nonlinearities involving convolution terms, namely
\begin{equation}\label{realnonlocal}
\begin{cases}
\partial_t u-\Delta u\geq \left(\mathcal{K}* |u|^{p_1}\right)|u|^{p_2}, &(t, x) \in(0, \infty) \times \mathbb{R}^n, \quad n \in \mathbb{N}^*, \\ u(0, x)=u_0(x), & x \in \mathbb{R}^n, \quad n \in \mathbb{N}^*,
\end{cases}
\end{equation}
were established by Filippucci and Ghergu in $2022$ \cite{Filippucci1, Filippucci2} for exponents $p_1, p_2>0$. Their
analysis in \cite{Filippucci1, Filippucci2} is carried out in a more general setting than \eqref{realnonlocal}, including quasilinear operators such as the $m$-Laplacian and the generalized mean curvature operator. In particular, \eqref{realnonlocal} admits no
nontrivial global weak solutions provided that $p_1+p_2>2, u_0\in L^1$,
$$
\int_{\mathbb{R}^n}u_0(x) dx>0 \quad \text{ and } \quad \limsup_{R\longrightarrow\infty}\mathcal{K}(R) R^{\frac{2n+2}{p_1+p_2}-n}>0.
$$
As a special case, they considered $\mathcal{K}(r)=r^{-(n-\gamma)}$ with $\gamma \in (n-1,n)$, and showed that \eqref{realnonlocal} has no global weak solutions provided that
$$
u_0\in L^1, \quad \int_{\mathbb{R}^n}u_0(x) dx>0 \quad \text{ and } \quad 2<p_1+p_2\leq 1+\frac{\gamma+2}{2n-\gamma}.
$$

\subsection{Main purpose}
We first study decay properties for the classical damped wave equation with initial data possessing additional negative regularity in homogeneous Besov spaces. More precisely, we consider
\begin{equation} \label{eq:1.1}
\begin{cases}
\partial_t^2 u-\Delta u+\partial_t u=0, & (t, x) \in(0, \infty) \times \mathbb{R}^n, \quad n \in \mathbb{N}^*, \\ u(0, x)=u_0(x), \quad \partial_t u(0, x)=u_1(x), & x \in\mathbb{R}^n, \quad n \in \mathbb{N}^*.
\end{cases}
\end{equation}
Our first aim is to derive decay estimates in $\dot{B}_{q,r}^{s}$ for initial data belonging to $\left(B_{q,r}^{s+\sigma_q}\cap\dot{B}_{m,\infty}^{-\beta}\right)
\times
\left(B_{q,r}^{s+\sigma_q-1}\cap\dot{B}_{m,\infty}^{-\beta}\right)$, where $1\leq m\leq q<\infty$ with $q>1$, $1\leq r\leq\infty$, $0\leq s<n/q$, $\beta\geq0$, $s+\beta+n(1/m-1/q)>0$, and $\sigma_q=(n-1)|1/2-1/q|$. Here the additional low-frequency information on the initial data is measured in $\dot{B}_{m,\infty}^{-\beta}$, whereas the inhomogeneous Besov spaces provide the regularity required in the high-frequency regime.

Decay estimates for damped wave equations with slowly decaying data and additional Besov regularity are available in several forms; see, for instance, \cite{Cardenas2022, Ikeda2026}. For our purposes, the relevant point is that the low-frequency part of the data is required only to belong to $\dot{B}_{m,\infty}^{-\beta}$; in particular, its third Besov index need not agree with the index $r$ of the target space $\dot{B}_{q,r}^{s}$. At high frequencies, we retain the known derivative loss $\sigma_q$. In the nonlinear argument, we use Proposition \ref{proposition:hilbert-linear} together with the $L^q\cap L^2$--$H^\alpha$ estimates for $K_1(t)*f$ proved in Section \ref{section_5}.

We then consider the semilinear Cauchy problem
\begin{equation} \label{eq:1.0}
\begin{cases}
\partial_t^2 u-\Delta u+\partial_t u=\mathcal{I}_\gamma\left(|u|^{p_1}\right)|u|^{p_2}, & (t, x) \in(0, \infty) \times \mathbb{R}^n, \quad n \in \mathbb{N}^*, \\ u(0, x)=u_0(x), \quad \partial_t u(0, x)=u_1(x), & x \in\mathbb{R}^n, \quad n \in \mathbb{N}^*,
\end{cases}
\end{equation}
where $p_1,p_2 \geq 1$, $\gamma\in(0,n)$, and the Riesz potential $\mathcal{I}_\gamma$ is defined in \eqref{eq:Riesz}.

Fino and Torebek \cite{Fino_Hartree} studied the corresponding parabolic problem and obtained nonexistence results for initial data satisfying a positive algebraic lower bound. D'Abbicco \cite{DAbbicco2023Hartree} proved global existence of small-data solutions in the supercritical range for an evolution equation with structural damping and Hartree-type nonlinearity. We consider the classical damped wave equation \eqref{eq:1.0} with initial data possessing additional regularity in $\dot{B}_{2,\infty}^{-\beta}$ and prove global existence under the assumptions of Theorem \ref{theorem:1.3}, including the critical case. Our proof uses the linear decay estimates, the strong Hardy-Littlewood-Sobolev inequality and the fractional Gagliardo-Nirenberg inequality. In the critical case, we estimate the nonlinear term in $L^{q_1}$ with $1<q_1<2n/(n+2\beta)$ over $[0,t/2]$ and obtain the decay rates of the linear problem. The test-function method then gives nonexistence of global weak solutions in the strictly subcritical range of Theorem \ref{theorem:1.5}. The data satisfying the lower-bound assumption in that theorem can be chosen arbitrarily small in the data space of Theorem \ref{theorem:1.3}.

More precisely, let $\alpha\in(0,1]$ and suppose that $\left(u_0,u_1\right)\in
\left(H^\alpha\cap\dot{B}_{2,\infty}^{-\beta}\right)
\times
\left(L^2\cap\dot{B}_{2,\infty}^{-\beta}\right)$ with $\beta\in [0, n/2)$. Under the assumptions on $p_1$ and $p_2$ in Theorem \ref{theorem:1.3}, the critical value is
$$
p_1+p_2
=
p_{\mathrm{Fuji}}\left(\frac{n+2\beta}{2+\gamma}\right)
:=
1+\frac{4+2\gamma}{n+2\beta},
\quad
\beta\in\left[0,\frac{n}{2}\right),
\quad
\gamma\in(0,n).
$$
If the critical line satisfies the remaining assumptions of Theorem \ref{theorem:1.3}, the equality case is contained in the global existence range. On the other hand, under the additional lower-bound assumption on the initial data, Theorem \ref{theorem:1.5} gives nonexistence of global weak solutions in the strictly subcritical interval $2<p_1+p_2<p_{\mathrm{Fuji}}((n+2\beta)/(2+\gamma))$.

We recall the homogeneous Besov spaces associated with a Littlewood-Paley decomposition. Let $\varphi$ be a smooth radial function supported in the annulus $\left\{\xi \in \mathbb{R}^n: 3/4 \leq|\xi| \leq 8/3\right\}$ such that $\sum_{j \in \mathbb{Z}} \varphi\left(2^{-j}\xi\right)=1$ for all $\xi \in \mathbb{R}^n\backslash\{0\}$. For $j \in \mathbb{Z}$, let $\dot{\Delta}_j$ be the homogeneous Littlewood-Paley projection defined by $\widehat{\dot{\Delta}_j f}(\xi)=\varphi\left(2^{-j}\xi\right)\widehat{f}(\xi)$; see \cite{Bahouri2011}. We work with the standard realization $\mathcal{S}'_h$ of homogeneous tempered distributions. For $s<n/p$ and $1 \leq p, q \leq \infty$, we set
$$
\dot{B}_{p, q}^s\left(\mathbb{R}^n\right):=\left\{f \in \mathcal{S}'_h\left(\mathbb{R}^n\right): \left(2^{j s}\left\|\dot{\Delta}_j f\right\|_{L^p}\right)_{j \in \mathbb{Z}} \in \ell^q(\mathbb{Z})\right\},
$$
with norm
$$
\left\|f\right\|_{\dot{B}_{p, q}^s}:=\left\|\left(2^{j s}\left\|\dot{\Delta}_j f\right\|_{L^p}\right)_{j \in \mathbb{Z}}\right\|_{\ell^q(\mathbb{Z})}.
$$
For negative regularity, the heat-semigroup characterization in \cite[Theorem 2.34]{Bahouri2011} gives, for $\delta>0$,
$$
f \in \dot{B}_{p, q}^{-\delta} \Longleftrightarrow t^{\frac{\delta}{2}}\left\|\mathrm{e}^{t\Delta}f\right\|_{L^p} \in L^q\left(\mathbb{R}^{+}, \frac{dt}{t}\right),
$$
with equivalent norms. Consequently, if $u_0 \in L^2 \cap \dot{B}_{2, \infty}^{-\delta}$, then
$$
\left\|\mathrm{e}^{t\Delta}u_0\right\|_{L^2} \lesssim (1+t)^{-\frac{\delta}{2}}\left(\left\|u_0\right\|_{L^2}+\left\|u_0\right\|_{\dot{B}_{2, \infty}^{-\delta}}\right).
$$

\noindent \textbf{Relations with other classes of initial data.}
Let $0\leq\beta<n/2$ and put $m:=2n/(n+2\beta)\in(1,2]$. By Lemma \ref{lemma:bernstein}, we have $L^m\hookrightarrow\dot{B}_{2,\infty}^{-\beta}$ with $\beta=n(1/m-1/2)$. The embeddings also give
$$
\dot{H}^{-\beta}=\dot{B}_{2,2}^{-\beta} \hookrightarrow\dot{B}_{2,\infty}^{-\beta}, \quad \dot{F}_{2,\rho}^{-\beta}\hookrightarrow\dot{B}_{2,\infty}^{-\beta}, \quad 1\leq\rho\leq\infty.
$$
Here the homogeneous Triebel--Lizorkin space is defined by
$$
\dot{F}_{2,\rho}^{-\beta}\left(\mathbb{R}^n\right):=\left\{f\in\mathcal{S}'_h\left(\mathbb{R}^n\right): \left\|\left(2^{-j\beta}\dot{\Delta}_j f\right)_{j\in\mathbb{Z}}\right\|_{\ell^\rho(\mathbb{Z})} \in L^2\left(\mathbb{R}^n\right)\right\}, \quad 1\leq\rho\leq\infty.
$$
The Besov framework is also closely related to the decay character used in the analysis of dissipative equations. Following Brandolese \cite{Brandolese2016}, for $f\in L^2$ and $r>-n/2$, we define the lower and upper decay indicators by
$$
P_r(f)_{-}:=\liminf_{\rho\to0^+}\rho^{-2r-n}\int_{|\xi|\leq\rho}|\widehat{f}(\xi)|^2d\xi,
\quad
P_r(f)_{+}:=\limsup_{\rho\to0^+}\rho^{-2r-n}\int_{|\xi|\leq\rho}|\widehat{f}(\xi)|^2d\xi.
$$
If there exists $r^*(f)\in(-n/2,\infty)$ such that $0<P_{r^*(f)}(f)_{-}\leq P_{r^*(f)}(f)_{+}<\infty$, we call $r^*(f)$ the decay character of $f$. For $\beta>0$, we introduce the subclass
$$
\dot{\mathcal{A}}_{2,\infty}^{-\beta}
:=
\left\{
f\in\dot{B}_{2,\infty}^{-\beta}:
\begin{array}{l}
\text{there exist }c>0, M\in\mathbb{N}\text{ and }(j_k)_{k\geq1}\subset\mathbb{Z}\\
\text{such that }j_k\to-\infty, |j_k-j_{k+1}|\leq M, 2^{-\beta j_k}\|\dot{\Delta}_{j_k}f\|_{L^2}\geq c
\end{array}
\right\}.
$$
With this notation, \cite[Proposition 4.3]{Brandolese2016} gives, for $f\in L^2$, the equivalence $r^*(f)=\beta-\frac{n}{2} \Longleftrightarrow f\in\dot{\mathcal{A}}_{2,\infty}^{-\beta} \subset\dot{B}_{2,\infty}^{-\beta}$. Moreover, \cite[Theorem 3.1]{Brandolese2016} yields the sharp two-sided decay $\left\|\mathrm{e}^{t\Delta}f\right\|_{L^2}
\sim(1+t)^{-\beta/2}$ for the corresponding heat flow. Hence $\dot{B}_{2,\infty}^{-\beta}$ provides the upper low-frequency control used in our arguments, while the Brandolese subclass singles out data whose heat flow decays exactly at the rate $(1+t)^{-\beta/2}$.

\subsection{Main results}
We now state the main results for the linear and semilinear damped wave equations.
\begin{theorem}[Linear decay estimates] \label{theorem:1.1}
Let $n \in \mathbb{N}^*$, $1 \leq m \leq q<\infty$ with $q>1$, $1 \leq r \leq\infty$, $0 \leq s<n/q$ and $\beta\geq0$. Put $\sigma_q:=(n-1)|1/2-1/q|$ and assume that $s+\beta+n(1/m-1/q)>0$. Let $\left(u_0,u_1\right)\in \left(B_{q,r}^{s+\sigma_q}\cap\dot{B}_{m,\infty}^{-\beta}\right) \times \left(B_{q,r}^{s+\sigma_q-1}\cap\dot{B}_{m,\infty}^{-\beta}\right)$, where $B_{q,r}^{\rho}$ denotes the usual inhomogeneous Besov space. Then the Fourier multiplier solution to \eqref{eq:1.1} satisfies
\begin{equation}\label{ineq:besov-linear}
\left\|u(t,\cdot{})\right\|_{\dot{B}_{q,r}^{s}} \lesssim (1+t)^{-\frac{s+\beta}{2}-\frac{n}{2}\left(\frac{1}{m}-\frac{1}{q}\right)} \Big(\left\|u_0\right\|_{B_{q,r}^{s+\sigma_q}}+\left\|u_0\right\|_{\dot{B}_{m,\infty}^{-\beta}}+\left\|u_1\right\|_{B_{q,r}^{s+\sigma_q-1}}+\left\|u_1\right\|_{\dot{B}_{m,\infty}^{-\beta}}\Big) \quad \text{ for all } t\geq0.
\end{equation}
If $\left(u_0,u_1\right)\in
\left(B_{q,r}^{s+1+\sigma_q}\cap\dot{B}_{m,\infty}^{-\beta}\right)
\times
\left(B_{q,r}^{s+\sigma_q}\cap\dot{B}_{m,\infty}^{-\beta}\right)$, then
\begin{equation}\label{ineq:besov-linear-time}
\left\|\partial_tu(t,\cdot{})\right\|_{\dot{B}_{q,r}^{s}} \lesssim (1+t)^{-\frac{s+\beta}{2}-\frac{n}{2}\left(\frac{1}{m}-\frac{1}{q}\right)-1}\Big(\left\|u_0\right\|_{B_{q,r}^{s+1+\sigma_q}}+\left\|u_0\right\|_{\dot{B}_{m,\infty}^{-\beta}}+\left\|u_1\right\|_{B_{q,r}^{s+\sigma_q}}+\left\|u_1\right\|_{\dot{B}_{m,\infty}^{-\beta}}\Big) \quad \text{ for all } t\geq0.
\end{equation}
\end{theorem}

\begin{remark}
\rm{
For $1\leq a\leq m$, put $\beta=n(1/a-1/m)$. By Lemma \ref{lemma:bernstein}, we have $L^a\hookrightarrow\dot{B}_{m,\infty}^{-\beta}$. Hence Theorem \ref{theorem:1.1} also applies to initial data in
$$
\left(B_{q,r}^{s+\sigma_q}\cap L^a\right)\times \left(B_{q,r}^{s+\sigma_q-1}\cap L^a\right)
$$
under its remaining assumptions. In this case, the decay factor in \eqref{ineq:besov-linear} becomes $(1+t)^{-s/2-(n/2)(1/a-1/q)}$.

For $m=2$, we also have
$$
\dot{H}^{-\beta}=\dot{B}_{2,2}^{-\beta}\hookrightarrow\dot{B}_{2,\infty}^{-\beta}, \quad \dot{F}_{2,\rho}^{-\beta}\hookrightarrow\dot{B}_{2,\infty}^{-\beta}, \quad 1\leq\rho\leq\infty.
$$
For $\beta>0$ and $f\in L^2$, the condition $r^*(f)=\beta-n/2$ is equivalent to $f\in\dot{\mathcal{A}}_{2,\infty}^{-\beta}\subset\dot{B}_{2,\infty}^{-\beta}$; see \cite[Proposition 4.3]{Brandolese2016}.
}
\end{remark}

\begin{proposition} \label{proposition:hilbert-linear}
Let $n\in\mathbb{N}^*$, $\alpha\in[0,1]$, $\beta\geq0$ and $\left(u_0,u_1\right)\in \left(H^\alpha\cap\dot{B}_{2,\infty}^{-\beta}\right) \times \left(L^2\cap\dot{B}_{2,\infty}^{-\beta}\right)$. Then the solution to \eqref{eq:1.1} satisfies $u\in\mathcal{C}\left([0,\infty),H^\alpha\right)$ and
\begin{equation} \label{ineq:1.3}
\left\|u(t,\cdot{})\right\|_{\dot{H}^\alpha} \lesssim(1+t)^{-\frac{\beta+\alpha}{2}} \left(\left\|u_0\right\|_{H^\alpha}+\left\|u_0\right\|_{\dot{B}_{2,\infty}^{-\beta}}+\left\|u_1\right\|_{L^2}+\left\|u_1\right\|_{\dot{B}_{2,\infty}^{-\beta}}\right).
\end{equation}
If $\left(u_0,u_1\right)\in \left(H^1\cap\dot{B}_{2,\infty}^{-\beta}\right)\times \left(L^2\cap\dot{B}_{2,\infty}^{-\beta}\right)$, then $u\in\mathcal{C}\left([0,\infty),H^1\right) \cap\mathcal{C}^1\left([0,\infty),L^2\right)$ and
\begin{equation} \label{ineq:1.4}
\left\|\partial_tu(t,\cdot{})\right\|_{L^2} \lesssim(1+t)^{-\frac{\beta}{2}-1} \left(\left\|u_0\right\|_{H^1}+\left\|u_0\right\|_{\dot{B}_{2,\infty}^{-\beta}}+\left\|u_1\right\|_{L^2}+\left\|u_1\right\|_{\dot{B}_{2,\infty}^{-\beta}}\right).
\end{equation}
\end{proposition}

We next state the small-data global existence result. A mild solution of \eqref{eq:1.0} is understood through the Duhamel formula in Section \ref{section_5}.
\begin{theorem}[Global existence] \label{theorem:1.3}
Let $n \in \mathbb{N}^*$, $\alpha \in (0, 1]$, $0 \leq \beta < n/2$, $0<\gamma<n$, and $\left(u_0, u_1\right) \in \mathcal{D}:=\left(H^\alpha \cap \dot{B}_{2, \infty}^{-\beta}\right) \times \left(L^2 \cap \dot{B}_{2, \infty}^{-\beta}\right)$. Let $p_1,p_2\geq1$ and set $p:=p_1+p_2$. Assume that
\begin{align}
&p_1>\frac{2\gamma}{n}, \quad p_2>\frac{2(\beta+\gamma)}{n}-1, \quad p_1+p_2>1+\frac{2(\beta+\gamma)}{n},\notag\\
\label{eq:time-threshold}
&p_1+p_2\geq p_{\mathrm{Fuji}}\left(\frac{n+2\beta}{2+\gamma}\right):=1+\frac{4+2\gamma}{n+2\beta}.
\end{align}
If $n>2\alpha$, assume in addition that
$$
p_1<\frac{2n}{n-2\alpha}, \quad p_2<\frac{n}{n-2\alpha}, \quad p_1+p_2\leq\frac{n+2\gamma}{n-2\alpha}.
$$
Then there exists a constant $\varepsilon>0$ such that, if $\left\|\left(u_0, u_1\right)\right\|_{\mathcal{D}}:=\left\|u_0\right\|_{H^\alpha}+\left\|u_0\right\|_{\dot{B}_{2, \infty}^{-\beta}}+\left\|u_1\right\|_{L^2}+\left\|u_1\right\|_{\dot{B}_{2, \infty}^{-\beta}} \leq \varepsilon$, problem \eqref{eq:1.0} has a unique global mild solution $u \in \mathcal{C}\left([0, \infty), H^\alpha\right)$. Moreover, for every $t\geq0$,
\begin{align}
\label{ineq:theorem:1.3.1} \left\|u(t, \cdot{})\right\|_{L^2} \lesssim& \left\|\left(u_0, u_1\right)\right\|_{\mathcal{D}}(1+t)^{-\frac{\beta}{2}},\\
\label{ineq:theorem:1.3.2} \left\|u(t, \cdot{})\right\|_{\dot{H}^\alpha} \lesssim& \left\|\left(u_0, u_1\right)\right\|_{\mathcal{D}}(1+t)^{-\frac{\beta+\alpha}{2}}.
\end{align}
\end{theorem}

\begin{remark}
\rm{
Let $0\leq\beta<n/2$ and put $m_\beta:=2n/(n+2\beta)\in(1,2]$. By Lemma \ref{lemma:bernstein} and the embeddings, we have
$$
L^{m_\beta} \hookrightarrow\dot{B}_{2,\infty}^{-\beta}, \quad \dot{H}^{-\beta}=\dot{B}_{2,2}^{-\beta} \hookrightarrow\dot{B}_{2,\infty}^{-\beta}, \quad \dot{F}_{2,\rho}^{-\beta} \hookrightarrow\dot{B}_{2,\infty}^{-\beta}, \quad 1\leq\rho\leq\infty.
$$
Under the parameter assumptions of Theorem \ref{theorem:1.3}, its conclusions hold if at least one of the following quantities is sufficiently small:
$$
\begin{aligned}
&\left\|u_0\right\|_{H^\alpha}+\left\|u_0\right\|_{L^{m_\beta}}+\left\|u_1\right\|_{L^2}+\left\|u_1\right\|_{L^{m_\beta}},\\
&\left\|u_0\right\|_{H^\alpha}+\left\|u_0\right\|_{\dot{H}^{-\beta}}+\left\|u_1\right\|_{L^2}+\left\|u_1\right\|_{\dot{H}^{-\beta}},\\
&\left\|u_0\right\|_{H^\alpha}+\left\|u_0\right\|_{\dot{F}_{2,\rho}^{-\beta}}+\left\|u_1\right\|_{L^2}+\left\|u_1\right\|_{\dot{F}_{2,\rho}^{-\beta}}.
\end{aligned}
$$
In all three cases, condition \eqref{eq:time-threshold} is
$$
p_1+p_2\geq1+\frac{4+2\gamma}{n+2\beta}=1+\frac{m_\beta(2+\gamma)}{n}.
$$

For $\beta>0$, suppose that
$$
\left(u_0,u_1\right)\in H^\alpha\times L^2, \quad r^*(u_0)=r^*(u_1)=\beta-\frac{n}{2}.
$$
By \cite[Proposition 4.3]{Brandolese2016}, we have $u_0,u_1\in\dot{B}_{2,\infty}^{-\beta}$. Hence $\left(u_0,u_1\right)\in\mathcal{D}$. Under the same parameter assumptions, Theorem \ref{theorem:1.3} gives a global solution with the stated decay estimates if $\left\|\left(u_0,u_1\right)\right\|_{\mathcal{D}}$ is sufficiently small.
}
\end{remark}

We next state the nonexistence result for global weak solutions.

\begin{theorem}[Nonexistence] \label{theorem:1.5}
Let $n \in \mathbb{N}^*$, $\gamma \in (0, n)$, $\beta \in [0, n/2)$ and $p_1,p_2\geq1$. Assume that $\left(u_0, u_1\right) \in \left(L_\mathrm{loc}^1 \cap \dot{B}_{2, \infty}^{-\beta}\right) \times \left(L_\mathrm{loc}^1 \cap \dot{B}_{2, \infty}^{-\beta}\right)$ and
\begin{equation}
\label{condition_u0u1}
u_0(x)+u_1(x) \gtrsim\langle x\rangle^{-\frac{n}{2}-\beta}\left(\log \left(\mathrm{e}+|x|\right)\right)^{-1} \quad \text{ for almost every } x \in \mathbb{R}^n,
\end{equation}
where $\langle x\rangle:=\left(1+|x|^2\right)^{\frac{1}{2}}$. Assume in addition that
\begin{equation} \label{p_critical}
2<p_1+p_2< p_{\mathrm{Fuji}}\left(\frac{n+2\beta}{2+\gamma}\right).
\end{equation}
Then \eqref{eq:1.0} admits no global weak solution in the sense of Definition \ref{def:weak}.
\end{theorem}

\begin{remark}
\rm{
For $p_1+p_2>2$, the mild solution obtained in Theorem \ref{theorem:1.3} is also a weak solution in the sense of Definition \ref{def:weak}. Let $T>0$. By the Sobolev embedding and the assumptions of Theorem \ref{theorem:1.3}, we have
$$
u\in L_\mathrm{loc}^{(p_1+p_2)/2}([0,T]\times\mathbb{R}^n).
$$
The proof of Theorem \ref{theorem:1.3} also gives
$$
F(u):=\mathcal{I}_\gamma(|u|^{p_1})|u|^{p_2} \in\mathcal{C}([0,T],L^2).
$$
We choose $\eta\in\mathcal{C}_0^\infty(\mathbb{R}^n)$ with $\eta\geq0$ and $\int_{\mathbb{R}^n}\eta(x)dx=1$. For $\varepsilon>0$, we put
$$
\eta_\varepsilon(x):=\varepsilon^{-n}\eta(x/\varepsilon), \quad u_\varepsilon:=\eta_\varepsilon*u, \quad F_\varepsilon:=\eta_\varepsilon*F(u).
$$
By the Duhamel formula, we have
$$
\begin{cases}
\partial_t^2u_\varepsilon-\Delta u_\varepsilon+\partial_tu_\varepsilon=F_\varepsilon,\\ u_\varepsilon(0)=\eta_\varepsilon*u_0, \quad \partial_tu_\varepsilon(0)=\eta_\varepsilon*u_1.
\end{cases}
$$
For every $\phi\in\mathcal{C}_0^\infty([0,T)\times\mathbb{R}^n)$,
integration by parts gives
$$
\begin{aligned}
&\int_0^T\int_{\mathbb{R}^n}F_\varepsilon\phi dxdt+\int_{\mathbb{R}^n}\left[\eta_\varepsilon*(u_0+u_1)\right] \phi(0,x) dx-\int_{\mathbb{R}^n}(\eta_\varepsilon*u_0)\partial_t\phi(0,x)dx \\
=&\int_0^T\int_{\mathbb{R}^n} u_\varepsilon(\partial_t^2\phi-\Delta\phi-\partial_t\phi) dxdt.
\end{aligned}
$$
As $\varepsilon\to0$, we have $u_\varepsilon\to u$ and
$F_\varepsilon\to F(u)$ in $\mathcal{C}([0,T],L^2)$,
and $\eta_\varepsilon*u_j\to u_j$ in $L^2$ for $j=0,1$.
Hence we obtain \eqref{eq:8.1}.
}
\end{remark}

\begin{remark}
\rm{
Let $m_\beta:=2n/(n+2\beta)>1$ and $C>0$. We take
$$
u_0(x)=u_1(x)=C\langle x\rangle^{-\frac{n}{2}-\beta}(\log(\mathrm{e}+|x|))^{-1}.
$$
Then \eqref{condition_u0u1} holds. Since $(n/2+\beta)m_\beta=n$ and $m_\beta>1$, we have
$$
\|u_j\|_{L^{m_\beta}}^{m_\beta}\lesssim C^{m_\beta}\left(1+\int_2^\infty r^{-1}(\log r)^{-m_\beta}dr\right)<\infty, \quad j=0,1.
$$
By Lemma \ref{lemma:bernstein}, we obtain $u_0,u_1\in L_\mathrm{loc}^1\cap\dot{B}_{2,\infty}^{-\beta}$. Moreover, for almost every $x\in\mathbb{R}^n$, we have
$$
|u_j(x)|+\langle x\rangle|\nabla u_j(x)|\lesssim C\langle x\rangle^{-\frac{n}{2}-\beta}(\log(\mathrm{e}+|x|))^{-1}, \quad j=0,1.
$$
Hence
$$
\|u_j\|_{H^1}^2\lesssim C^2\left(1+\int_2^\infty r^{-1-2\beta}(\log r)^{-2}dr\right)<\infty, \quad j=0,1.
$$
Thus, for every $\alpha\in(0,1]$, we have
$$
(u_0,u_1)\in\mathcal{D}, \quad \|(u_0,u_1)\|_{\mathcal{D}}\lesssim C,
$$
where $\mathcal{D}$ is defined in Theorem \ref{theorem:1.3}. Hence we can choose the initial data arbitrarily small.
}
\end{remark}

\begin{remark}
\rm{
Put $P:=p_{\mathrm{Fuji}}((n+2\beta)/(2+\gamma))$ and assume that
$$
P>1+\frac{2(\beta+\gamma)}{n}.
$$
We consider the following cases.
\begin{itemize}
\item If $n\leq2\alpha$, choose $p_1$ such that
$$
\max\left\{1,\frac{2\gamma}{n}\right\}<p_1<P-\max\left\{1,\frac{2(\beta+\gamma)}{n}-1\right\}.
$$
\item If $n>2\alpha$ and $P\leq(n+2\gamma)/(n-2\alpha)$, choose $p_1$ such that
$$
\max\left\{1,\frac{2\gamma}{n},P-\frac{n}{n-2\alpha}\right\}<p_1<\min\left\{\frac{2n}{n-2\alpha},P-1,P-\frac{2(\beta+\gamma)}{n}+1\right\}.
$$
\end{itemize}
If either interval is nonempty, we set $p_2:=P-p_1$. Then $p_1,p_2$ satisfy all the assumptions of Theorem \ref{theorem:1.3} with $p_1+p_2=P$.

On the other hand, Theorem \ref{theorem:1.5} gives nonexistence of global weak solutions for $2<p_1+p_2<P$ when the initial data satisfy \eqref{condition_u0u1}. By the preceding remark, we can choose such data arbitrarily small in $\mathcal{D}$.
}
\end{remark}

\begin{remark}
\rm{
Let $\alpha=1$ and put
$$
P:=p_{\mathrm{Fuji}}\left(\frac{n+2\beta}{2+\gamma}\right)=1+\frac{4+2\gamma}{n+2\beta}.
$$
We have
$$
P>2 \Longleftrightarrow n+2\beta<4+2\gamma \quad \text{ and } \quad P>1+\frac{2(\beta+\gamma)}{n} \Longleftrightarrow \beta(n+2\beta+2\gamma)<2n.
$$
For $n\geq3$, the first condition also gives $P\leq (n+2\gamma)/(n-2)$.

We consider the following parameter ranges.
\begin{itemize}
\item If $n=1$, take $0\leq\beta<1/2$ and $0<\gamma<1$.

\item If $n=2$, take $0\leq\beta<1$, $0<\gamma<2$ and $\beta(1+\beta+\gamma)<2$.

\item If $n=3$, take $0\leq\beta<3/2$, $\max\{0,\beta-1/2\}<\gamma<3$ and $\beta(3+2\beta+2\gamma)<6$.

\item If $n=4$, take $0\leq\beta<2$, $\beta<\gamma<4$ and $\beta(4+2\beta+2\gamma)<8$.

\item If $n\geq5$, take $0\leq\beta<n/2$, $n/2+\beta-2<\gamma<n$ and $\beta(n+2\beta+2\gamma)<2n$.
\end{itemize}
For each of these cases, the assumptions of Theorem \ref{theorem:1.3} admit $p_1,p_2\geq1$ with $p_1+p_2=P$. More precisely, the theorem gives global existence for sufficiently small data when
$$
p_1>\frac{2\gamma}{n}, \quad p_2>\frac{2(\beta+\gamma)}{n}-1, \quad p_1+p_2\geq P,
$$
and, for $n\geq3$,
$$
p_1<\frac{2n}{n-2}, \quad p_2<\frac{n}{n-2}, \quad p_1+p_2\leq\frac{n+2\gamma}{n-2}.
$$
On the other hand, Theorem \ref{theorem:1.5} gives nonexistence of global weak solutions for
$$
p_1,p_2\geq1, \quad 2<p_1+p_2<P,
$$
under the assumption \eqref{condition_u0u1}. Thus, for every $n\geq1$, the critical existence range and the strictly subcritical nonexistence range are nonempty.
}
\end{remark}

\noindent \textbf{Notations.}
We use $\mathbb{N}^*:=\{1,2,\ldots\}$. We write $f\lesssim g$ when there exists a constant $C>0$ such that $f\leq Cg$, set $f\gtrsim g$ if $g\lesssim f$, and write $f\sim g$ if both $f\lesssim g$ and $f\gtrsim g$ hold. Moreover, $X=O(Y)$ means that $|X|\leq C|Y|$ in the asymptotic regime under consideration, with $C>0$ independent of the limiting variable. For brevity, we write $L_\mathrm{loc}^1, H^s, \dot{H}^s, L^p, \dot{B}_{p, q}^s$ and $\ell^q$ instead of $L_\mathrm{loc}^1\left(\mathbb{R}^n\right), H^s\left(\mathbb{R}^n\right), \dot{H}^s\left(\mathbb{R}^n\right), L^p\left(\mathbb{R}^n\right), \dot{B}_{p, q}^s\left(\mathbb{R}^n\right)$ and $\ell^q(\mathbb{Z})$, respectively. As usual, the spaces $H^s$ and $\dot{H}^s$ with $s \geq 0$ stand for Bessel and Riesz potential spaces based on $L^2$. Negative-order homogeneous Sobolev spaces are understood in the standard homogeneous distributional realization; in particular, $\dot{H}^{-\beta}=\dot{B}_{2,2}^{-\beta}$. For two Banach spaces $X$ and $Y$, we use $\left\|f\right\|_{X\cap Y}:=\left\|f\right\|_X+\left\|f\right\|_Y$.

\noindent \textbf{The paper is organized as follows.} In Section \ref{section_4}, we prove Theorem \ref{theorem:1.1} and Proposition \ref{proposition:hilbert-linear}. In Section \ref{section_5}, we prove the global small-data existence result in Theorem \ref{theorem:1.3}. Section \ref{section_6} is devoted to the nonexistence result in Theorem \ref{theorem:1.5}. Finally, Section \ref{section_tools} collects the tools from harmonic analysis used in the proofs.

\section{Linear estimates} \label{section_4}
In this section, we derive decay estimates for the linear Cauchy problem \eqref{eq:1.1} with initial data in homogeneous Besov spaces. Let $u$ denote the solution to \eqref{eq:1.1}. Taking the Fourier transform in $x$ of \eqref{eq:1.1}, we obtain the following parameter-dependent ordinary differential equation for $\widehat{u}(t, \xi)$
$$
\begin{cases}
\partial_t^2 \widehat{u}+\partial_t \widehat{u}+|\xi|^2 \widehat{u}=0, \quad t>0, \\
\widehat{u}(0, \xi)=\widehat{u}_0(\xi), \quad \partial_t \widehat{u}(0, \xi)=\widehat{u}_1(\xi).
\end{cases}
$$
The roots of the characteristic equation $\tau^2+\tau+|\xi|^2=0$ are
$$
\tau_{ \pm}= \begin{cases}-\frac{1}{2} \pm i \sqrt{|\xi|^2-\frac{1}{4}} & \text { if } |\xi|>\frac{1}{2}, \\
-\frac{1}{2} & \text { if } |\xi|=\frac{1}{2}, \\
-\frac{1}{2} \pm \sqrt{\frac{1}{4}-|\xi|^2} & \text { if } |\xi|<\frac{1}{2}.
\end{cases}
$$
A direct computation gives the representation formula
\begin{equation}\label{eq:4.4}
\widehat{u}(t, \xi)=\widehat{u}_0(\xi) \widehat{K}_0(t, \xi)+\left(\frac{1}{2} \widehat{u}_0(\xi)+\widehat{u}_1(\xi)\right) \widehat{K}_1(t, \xi)=\left(\widehat{K}_0(t, \xi)+\frac{1}{2}\widehat{K}_1(t, \xi)\right)\widehat{u}_0(\xi)+\widehat{K}_1(t, \xi) \widehat{u}_1(\xi),
\end{equation}
where $\widehat{K}_0$ satisfies
\begin{equation} \label{eq:F}
\begin{cases}
\partial_t^2 \widehat{K}_0(t, \xi)+\partial_t \widehat{K}_0(t, \xi)+|\xi|^2 \widehat{K}_0(t, \xi)=0, \\
\widehat{K}_0(0, \xi)=1, \quad \partial_t \widehat{K}_0(0, \xi)=-1/2,
\end{cases}
\end{equation}
and $\widehat{K}_1$ satisfies
\begin{equation} \label{eq:G}
\begin{cases}
\partial_t^2 \widehat{K}_1(t, \xi)+\partial_t \widehat{K}_1(t, \xi)+|\xi|^2 \widehat{K}_1(t, \xi)=0, \\
\widehat{K}_1(0, \xi)=0, \quad \partial_t \widehat{K}_1(0, \xi)=1.
\end{cases}
\end{equation}
Solving \eqref{eq:F} and \eqref{eq:G} explicitly, we obtain
$$
\widehat{K}_0(t, \xi):=
\begin{cases}
\mathrm{e}^{-t/2}\cosh\left(\frac{\sqrt{1-4|\xi|^2}}{2}t\right), & |\xi|<\frac{1}{2},\\
\mathrm{e}^{-t/2}, & |\xi|=\frac{1}{2},\\
\mathrm{e}^{-t/2}\cos\left(\frac{\sqrt{4|\xi|^2-1}}{2}t\right), & |\xi|>\frac{1}{2},
\end{cases}
\quad
\widehat{K}_1(t, \xi):=
\begin{cases}
\frac{2\mathrm{e}^{-t/2}}{\sqrt{1-4|\xi|^2}}\sinh\left(\frac{\sqrt{1-4|\xi|^2}}{2}t\right), & |\xi|<\frac{1}{2},\\
t\mathrm{e}^{-t/2}, & |\xi|=\frac{1}{2},\\
\frac{2\mathrm{e}^{-t/2}}{\sqrt{4|\xi|^2-1}}\sin\left(\frac{\sqrt{4|\xi|^2-1}}{2}t\right), & |\xi|>\frac{1}{2}.
\end{cases}
$$
\begin{proof}[Proof of Theorem \ref{theorem:1.1}]
By \eqref{eq:4.4}, we have
$$
u(t)=\big(K_0(t)+\frac{1}{2}K_1(t)\big)*u_0+K_1(t)*u_1.
$$
We first take $j\leq-3$. We put $\xi=2^j\eta$ and
$$
a_j(\eta)=\frac{2|\eta|^2}{1+\sqrt{1-4\cdot2^{2j}|\eta|^2}}, \quad b_j(\eta)=1-2^{2j}a_j(\eta), \quad d_j(\eta)=\sqrt{1-4\cdot2^{2j}|\eta|^2}.
$$
We choose $\widetilde\varphi\in C_c^\infty(\mathbb{R}^n\setminus\{0\})$ such that $\widetilde\varphi=1$ on $\operatorname{supp}\varphi$ and $\operatorname{supp}\widetilde\varphi\subset\{1/2<|\eta|<3\}$. For $j\leq-3$ and $\eta\in\operatorname{supp}\widetilde\varphi$, we have
$$
d_j(\eta)\geq\frac{\sqrt7}{4}, \quad b_j(\eta)=\frac{1+d_j(\eta)}{2}, \quad a_j(\eta)\sim b_j(\eta)\sim1.
$$
For every multi-index $\kappa$ with $|\kappa|\geq1$, we also have
$$
|\partial_\eta^\kappa a_j|\lesssim1, \quad |\partial_\eta^\kappa b_j|+|\partial_\eta^\kappa d_j|+|\partial_\eta^\kappa d_j^{-1}|\lesssim2^{2j}.
$$
Since $\tau_+(2^j\eta)=-2^{2j}a_j(\eta)$ and $\tau_-(2^j\eta)=-b_j(\eta)$, we obtain
$$
\widehat{K}_1(t,2^j\eta)=\frac{\mathrm{e}^{-t2^{2j}a_j(\eta)}-\mathrm{e}^{-tb_j(\eta)}}{d_j(\eta)}, \quad \widehat{K}_0(t,2^j\eta)+\frac{1}{2}\widehat{K}_1(t,2^j\eta)=\frac{b_j(\eta)\mathrm{e}^{-t2^{2j}a_j(\eta)}-2^{2j}a_j(\eta)\mathrm{e}^{-tb_j(\eta)}}{d_j(\eta)}.
$$
We put $N=[n/2]+1$. By the chain rule and $x^\ell \mathrm{e}^{-cx}\lesssim \mathrm{e}^{-c'x}$ for $x\geq0$, we get
$$
|\partial_\eta^\kappa \mathrm{e}^{-t2^{2j}a_j(\eta)}|\lesssim \mathrm{e}^{-ct2^{2j}}, \quad |\partial_\eta^\kappa \mathrm{e}^{-tb_j(\eta)}|\lesssim \mathrm{e}^{-ct}, \quad |\kappa|\leq N.
$$
We also use
$$
\partial_t\big(\widehat{K}_0(t,2^j\eta)+\frac{1}{2}\widehat{K}_1(t,2^j\eta)\big)=-2^{2j}|\eta|^2\widehat{K}_1(t,2^j\eta), \quad \partial_t\widehat{K}_1(t,2^j\eta)=\frac{-2^{2j}a_j(\eta)\mathrm{e}^{-t2^{2j}a_j(\eta)}+b_j(\eta)\mathrm{e}^{-tb_j(\eta)}}{d_j(\eta)}.
$$
Thus, for $|\kappa|\leq N$, we have
$$
\begin{aligned}
|\partial_\eta^\kappa(\widehat{K}_0(t,2^j\eta)+\frac{1}{2}\widehat{K}_1(t,2^j\eta))|+|\partial_\eta^\kappa\widehat{K}_1(t,2^j\eta)|\lesssim & \mathrm{e}^{-ct2^{2j}},\\
|\partial_\eta^\kappa\partial_t(\widehat{K}_0(t,2^j\eta)+\frac{1}{2}\widehat{K}_1(t,2^j\eta))|\lesssim& 2^{2j}\mathrm{e}^{-ct2^{2j}},\\
|\partial_\eta^\kappa\partial_t\widehat{K}_1(t,2^j\eta)|\lesssim& 2^{2j}\mathrm{e}^{-ct2^{2j}}+\mathrm{e}^{-ct}.
\end{aligned}
$$
For $1<q<\infty$ and $j\leq-3$, applying the Mikhlin multiplier theorem with the cutoff $\widetilde\varphi(\eta)$ and using $\xi=2^j\eta$, we obtain
\begin{align}
\|\dot{\Delta}_j u(t)\|_{L^q}\lesssim& \mathrm{e}^{-ct2^{2j}}\big(\|\dot{\Delta}_j u_0\|_{L^q}+\|\dot{\Delta}_j u_1\|_{L^q}\big),\label{ineq:dyadic-low-Lq}\\
\|\dot{\Delta}_j\partial_tu(t)\|_{L^q}\lesssim&\big(2^{2j}\mathrm{e}^{-ct2^{2j}}+\mathrm{e}^{-ct}\big)\big(\|\dot{\Delta}_j u_0\|_{L^q}+\|\dot{\Delta}_j u_1\|_{L^q}\big).\label{ineq:dyadic-low-Lq-time}
\end{align}
For $m\leq q$, Lemma \ref{lemma:bernstein} gives
\begin{equation}\label{ineq:bernstein-low}
\|\dot{\Delta}_j f\|_{L^q}\lesssim2^{jn(1/m-1/q)}\|\dot{\Delta}_j f\|_{L^m}\lesssim2^{j(\beta+n(1/m-1/q))}\|f\|_{\dot{B}_{m,\infty}^{-\beta}}.
\end{equation}

For $-2\leq j\leq0$, we have $3/16\leq|\xi|\leq8/3$ on $\operatorname{supp}\varphi(2^{-j}\xi)$. We choose $\delta>0$ such that
$$
|1/4-|\xi|^2|\leq1/16 \quad \text{ if } ||\xi|-1/2|\leq2\delta.
$$
For $||\xi|-1/2|\geq\delta$, the formulas for $\widehat{K}_0$ and $\widehat{K}_1$ give
$$
|\partial_\xi^\kappa\partial_t^\nu\widehat{K}_i(t,\xi)|\lesssim(1+t)^{N_\kappa}\mathrm{e}^{-ct}, \quad |\kappa|\leq N, \quad \nu=0,1, \quad i=0,1.
$$
For $||\xi|-1/2|\leq2\delta$, we use
$$
\widehat{K}_0(t,\xi)=\mathrm{e}^{-t/2}\sum_{\ell=0}^\infty\frac{t^{2\ell}(1/4-|\xi|^2)^\ell}{(2\ell)!}, \quad \widehat{K}_1(t,\xi)=\mathrm{e}^{-t/2}\sum_{\ell=0}^\infty\frac{t^{2\ell+1}(1/4-|\xi|^2)^\ell}{(2\ell+1)!}.
$$
We obtain
$$
|\partial_\xi^\kappa\partial_t^\nu\widehat{K}_i(t,\xi)|\lesssim(1+t)^{2N+2}\mathrm{e}^{-t/4}\lesssim \mathrm{e}^{-ct}, \quad |\kappa|\leq N, \quad \nu=0,1, \quad i=0,1.
$$
By the Mikhlin multiplier theorem, we get
\begin{equation}\label{ineq:dyadic-middle}
\|\dot{\Delta}_j u(t)\|_{L^q}+\|\dot{\Delta}_j\partial_tu(t)\|_{L^q}\lesssim \mathrm{e}^{-ct}\big(\|\dot{\Delta}_j u_0\|_{L^q}+\|\dot{\Delta}_j u_1\|_{L^q}\big), \quad -2\leq j\leq0.
\end{equation}

For $j\geq1$, we apply \cite[Proposition 2.2 and Remark 2.3(iii)]{Ikeda2026}, using the estimates in \cite{Ikeda2019}. With $\sigma_q=(n-1)|1/2-1/q|$ and $\widehat{\mathcal{D}}(t,\xi)=\widehat{K}_1(t,\xi)$, we obtain
$$
\left\|\big(2^{js}\|\dot{\Delta}_j(K_1(t)*g)\|_{L^q}\big)_{j\geq1}\right\|_{\ell^r}\lesssim \mathrm{e}^{-ct}\|g\|_{B_{q,r}^{s+\sigma_q-1}}, \quad \left\|\big(2^{js}\|\dot{\Delta}_j(\partial_tK_1(t)*g)\|_{L^q}\big)_{j\geq1}\right\|_{\ell^r}\lesssim \mathrm{e}^{-ct}\|g\|_{B_{q,r}^{s+\sigma_q}}.
$$
Since $K_0+K_1/2=(\partial_t+1)K_1$, we have
\begin{equation}\label{ineq:dyadic-high-Lq}
\left\|\big(2^{js}\|\dot{\Delta}_j u(t)\|_{L^q}\big)_{j\geq1}\right\|_{\ell^r}\lesssim \mathrm{e}^{-ct}\big(\|u_0\|_{B_{q,r}^{s+\sigma_q}}+\|u_1\|_{B_{q,r}^{s+\sigma_q-1}}\big).
\end{equation}
We also have $\partial_tu=\Delta K_1(t)*u_0+\partial_tK_1(t)*u_1$. By the estimates for $K_1$ and $\partial_tK_1$, we obtain
\begin{align}\label{ineq:dyadic-high-Lq-time}
\left\|\big(2^{js}\|\dot{\Delta}_j\partial_tu(t)\|_{L^q}\big)_{j\geq1}\right\|_{\ell^r}\lesssim & \left\|\big(2^{j(s+2)}\|\dot{\Delta}_j(K_1(t)*u_0)\|_{L^q}\big)_{j\geq1}\right\|_{\ell^r}+\left\|\big(2^{js}\|\dot{\Delta}_j(\partial_tK_1(t)*u_1)\|_{L^q}\big)_{j\geq1}\right\|_{\ell^r} \notag\\
\lesssim& \mathrm{e}^{-ct} \big(\|u_0\|_{B_{q,r}^{s+1+\sigma_q}} +\|u_1\|_{B_{q,r}^{s+\sigma_q}}\big).
\end{align}

We put $\Lambda=s+\beta+n(1/m-1/q)>0$. By \eqref{ineq:dyadic-low-Lq} and \eqref{ineq:bernstein-low}, we have
$$
2^{js}\|\dot{\Delta}_j u(t)\|_{L^q}\lesssim \mathrm{e}^{-ct2^{2j}}2^{j\Lambda}\big(\|u_0\|_{\dot{B}_{m,\infty}^{-\beta}}+\|u_1\|_{\dot{B}_{m,\infty}^{-\beta}}\big), \quad j\leq-3.
$$
For $\Lambda>0$ and $1\leq r\leq\infty$, we prove
\begin{equation}\label{ineq:dyadic-sum}
\left\|\big(\mathrm{e}^{-ct2^{2j}}2^{j\Lambda}\big)_{j\leq-3}\right\|_{\ell^r}\lesssim(1+t)^{-\Lambda/2}.
\end{equation}
For $0\leq t\leq64$, we use $\sum_{j\leq-3}2^{jr\Lambda}<\infty$ if $r<\infty$ and $\sup_{j\leq-3}2^{j\Lambda}<\infty$ if $r=\infty$. For $t>64$, we choose $j_t\leq-3$ such that $2^{2j_t}\leq t^{-1}<2^{2(j_t+1)}$. If $1\leq r<\infty$, we obtain
$$
\sum_{j\leq j_t}\mathrm{e}^{-crt2^{2j}}2^{jr\Lambda}\lesssim 2^{j_tr\Lambda}\lesssim t^{-r\Lambda/2}, \quad \sum_{j_t<j\leq-3}\mathrm{e}^{-crt2^{2j}}2^{jr\Lambda}\lesssim 2^{j_tr\Lambda}\sum_{k\geq1}\mathrm{e}^{-c_14^k}2^{kr\Lambda}\lesssim t^{-r\Lambda/2}.
$$
For $r=\infty$, we also have
$$
\sup_{j\leq j_t}\mathrm{e}^{-ct2^{2j}}2^{j\Lambda}\lesssim 2^{j_t\Lambda}, \quad \sup_{j_t<j\leq-3}\mathrm{e}^{-ct2^{2j}}2^{j\Lambda}\lesssim 2^{j_t\Lambda}\sup_{k\geq1}\mathrm{e}^{-c_14^k}2^{k\Lambda}\lesssim t^{-\Lambda/2}.
$$
This proves \eqref{ineq:dyadic-sum}.

By \eqref{ineq:dyadic-low-Lq}, \eqref{ineq:bernstein-low} and \eqref{ineq:dyadic-sum}, we obtain
$$
\left\|\big(2^{js}\|\dot{\Delta}_j u(t)\|_{L^q}\big)_{j\leq-3}\right\|_{\ell^r}\lesssim(1+t)^{-\Lambda/2}\big(\|u_0\|_{\dot{B}_{m,\infty}^{-\beta}}+\|u_1\|_{\dot{B}_{m,\infty}^{-\beta}}\big).
$$
Using \eqref{ineq:dyadic-middle} and \eqref{ineq:dyadic-high-Lq}, we obtain \eqref{ineq:besov-linear}, since $\mathrm{e}^{-ct}\lesssim(1+t)^{-\Lambda/2}$.

For the time derivative, \eqref{ineq:dyadic-low-Lq-time} and \eqref{ineq:bernstein-low} give
$$
2^{js}\|\dot{\Delta}_j\partial_tu(t)\|_{L^q}\lesssim\big(2^{j(\Lambda+2)}\mathrm{e}^{-ct2^{2j}}+2^{j\Lambda}\mathrm{e}^{-ct}\big)\big(\|u_0\|_{\dot{B}_{m,\infty}^{-\beta}}+\|u_1\|_{\dot{B}_{m,\infty}^{-\beta}}\big), \quad j\leq-3.
$$
By \eqref{ineq:dyadic-sum}, we have
$$
\left\|\big(\mathrm{e}^{-ct2^{2j}}2^{j(\Lambda+2)}\big)_{j\leq-3}\right\|_{\ell^r} \lesssim(1+t)^{-\frac{\Lambda+2}{2}}=(1+t)^{-\frac{\Lambda}{2}-1}.
$$
Since $\Lambda>0$, we also have
$$
\left\|\big(\mathrm{e}^{-ct}2^{j\Lambda}\big)_{j\leq-3}\right\|_{\ell^r}=\mathrm{e}^{-ct}\left\|\big(2^{j\Lambda}\big)_{j\leq-3}\right\|_{\ell^r}\lesssim \mathrm{e}^{-ct}\lesssim(1+t)^{-\frac{\Lambda}{2}-1}.
$$
Hence
$$
\left\|\big(2^{js}\|\dot{\Delta}_j\partial_tu(t)\|_{L^q}\big)_{j\leq-3}\right\|_{\ell^r}\lesssim(1+t)^{-\Lambda/2-1}\big(\|u_0\|_{\dot{B}_{m,\infty}^{-\beta}}+\|u_1\|_{\dot{B}_{m,\infty}^{-\beta}}\big).
$$
Using \eqref{ineq:dyadic-middle} and \eqref{ineq:dyadic-high-Lq-time}, we obtain \eqref{ineq:besov-linear-time}.
This completes the proof of Theorem \ref{theorem:1.1}.
\end{proof}

\begin{proof}[Proof of Proposition \ref{proposition:hilbert-linear}]
For $\alpha+\beta>0$, applying \eqref{ineq:dyadic-low-Lq} with $m=q=2$ and \eqref{ineq:dyadic-sum} with $\Lambda=\alpha+\beta$, we obtain
$$
\left\|\big(2^{j\alpha}\|\dot{\Delta}_j u(t)\|_{L^2}\big)_{j\leq-3}\right\|_{\ell^2}\lesssim(1+t)^{-(\alpha+\beta)/2}\big(\|u_0\|_{\dot{B}_{2,\infty}^{-\beta}}+\|u_1\|_{\dot{B}_{2,\infty}^{-\beta}}\big).
$$
Using \eqref{ineq:dyadic-middle} and \eqref{ineq:dyadic-high-Lq} with $q=r=2$ and $\sigma_2=0$, we also have
$$
\left\|\big(2^{j\alpha}\|\dot{\Delta}_j u(t)\|_{L^2}\big)_{j\geq-2}\right\|_{\ell^2}\lesssim \mathrm{e}^{-ct}\big(\|u_0\|_{H^\alpha}+\|u_1\|_{L^2}\big),
$$
since $\alpha\leq1$. Thus we obtain \eqref{ineq:1.3} for $\alpha+\beta>0$.

For $\alpha=\beta=0$, using the formulas for $\widehat{K}_0$ and $\widehat{K}_1$ on $|\xi|\leq1/4$, $1/4\leq|\xi|\leq1$ and $|\xi|\geq1$, we obtain
$$
\sup_{t\geq0, \xi\in\mathbb{R}^n}\big(|\widehat{K}_0(t,\xi)+\frac{1}{2}\widehat{K}_1(t,\xi)|+|\widehat{K}_1(t,\xi)|\big)<\infty.
$$
By Plancherel's identity, we obtain \eqref{ineq:1.3} also in this case.

For \eqref{ineq:1.4}, we first take $\beta>0$. By \eqref{ineq:dyadic-low-Lq-time} with $m=q=2$ and \eqref{ineq:dyadic-sum} with $\Lambda=\beta+2$, we get
$$
\left\|\big(\|\dot{\Delta}_j\partial_tu(t)\|_{L^2}\big)_{j\leq-3}\right\|_{\ell^2}\lesssim(1+t)^{-\beta/2-1}\big(\|u_0\|_{\dot{B}_{2,\infty}^{-\beta}}+\|u_1\|_{\dot{B}_{2,\infty}^{-\beta}}\big),
$$
where we use $\sum_{j\leq-3}2^{2j\beta}<\infty$ for the term containing $\mathrm{e}^{-ct}$. If $\beta=0$, we have
$$
2^{2j}\mathrm{e}^{-ct2^{2j}}+\mathrm{e}^{-ct}\lesssim(1+t)^{-1}, \quad j\leq-3.
$$
Applying \eqref{ineq:dyadic-low-Lq-time} and Plancherel's identity, we obtain
$$
\left\|\big(\|\dot{\Delta}_j\partial_tu(t)\|_{L^2}\big)_{j\leq-3}\right\|_{\ell^2}\lesssim(1+t)^{-1}\big(\|u_0\|_{L^2}+\|u_1\|_{L^2}\big).
$$
Using \eqref{ineq:dyadic-middle} and \eqref{ineq:dyadic-high-Lq-time} with $q=r=2$, we obtain \eqref{ineq:1.4} for every $\beta\geq0$.

We finally prove $u\in C([0,\infty),H^\alpha)$ and $\partial_tu\in C([0,\infty),L^2)$ under the assumptions of Proposition \ref{proposition:hilbert-linear}. Let $T>0$. For $0\leq\alpha\leq1$, the formulas for $\widehat{K}_0$ and $\widehat{K}_1$ give
$$
\sup_{0\leq t\leq T, \xi\in\mathbb{R}^n}\big(|\widehat{K}_0(t,\xi)+\frac{1}{2}\widehat{K}_1(t,\xi)|+\langle\xi\rangle^\alpha|\widehat{K}_1(t,\xi)|\big)<\infty.
$$
For every $\xi\in\mathbb{R}^n$ and $i=0,1$, we have $\widehat{K}_i(t,\xi)\to\widehat{K}_i(t_0,\xi)$ as $t\to t_0$. By Plancherel's identity and dominated convergence, we obtain $u\in C([0,\infty),H^\alpha)$. If $u_0\in H^1$, we also have
$$
\sup_{0\leq t\leq T, \xi\in\mathbb{R}^n}\left(\frac{|\xi|^2|\widehat{K}_1(t,\xi)|}{\langle\xi\rangle}+|\partial_t\widehat{K}_1(t,\xi)|\right)<\infty.
$$
Using $\partial_t(\widehat{K}_0+\widehat{K}_1/2)=-|\xi|^2\widehat{K}_1$ and dominated convergence, we obtain $\partial_tu\in C([0,\infty),L^2)$.
\end{proof}

\section{Global existence} \label{section_5}
We write the solution as the sum of its linear part
$$
u^{\mathrm{lin}}(t,x):=\left[\left(K_0(t,\cdot{})+\frac{1}{2}K_1(t,\cdot{})\right)*u_0\right](x)+\left[K_1(t,\cdot{})*u_1\right](x),
$$
and its Duhamel part
$$
u^{\mathrm{non}}(t,x):=\int_0^t\left[K_1(t-s,\cdot{})*\left(\mathcal{I}_\gamma\left(|u(s,\cdot{})|^{p_1}\right)|u(s,\cdot{})|^{p_2}\right)\right](x)ds.
$$
For $T>0$, write $Nu:=u^{\mathrm{lin}}+u^{\mathrm{non}}$. The evolution space $X(T)$ is introduced at the beginning of the proof. We first verify that the Duhamel term is well defined in $H^\alpha$ and that $N$ maps $X(T)$ into itself. To prove Theorem \ref{theorem:1.3}, it is enough to establish
\begin{align} \label{ineq:5.1}
\left\|Nu\right\|_{X(T)} \lesssim & \left\|\left(u_0,u_1\right)\right\|_{\mathcal{D}}+\left\|u\right\|_{X(T)}^{p},\\
\label{ineq:5.2}
\left\|Nu-Nv\right\|_{X(T)} \lesssim & \left\|u-v\right\|_{X(T)}\left(\left\|u\right\|_{X(T)}^{p-1}+\left\|v\right\|_{X(T)}^{p-1}\right).
\end{align}

\begin{proof}[Proof of Theorem \ref{theorem:1.3}]
For $T>0$, we set $X(T)=\mathcal{C}([0,T],H^\alpha)$ with the norm
$$
\|u\|_{X(T)}:=\sup_{0\leq t\leq T}\Big[(1+t)^{\frac{\beta}{2}}\|u(t)\|_{L^2}+(1+t)^{\frac{\beta+\alpha}{2}}\|u(t)\|_{\dot{H}^\alpha}\Big].
$$
For every $T>0$, we have
$$
\|u\|_{X(T)}\sim\|u\|_{\mathcal{C}([0,T],H^\alpha)}.
$$
Hence $X(T)$ is a Banach space. Applying \eqref{ineq:1.3} with $\alpha=0$ and then with the given $\alpha>0$, we obtain
\begin{align}
\|u^{\mathrm{lin}}(t)\|_{L^2}\lesssim&(1+t)^{-\frac{\beta}{2}}\|(u_0,u_1)\|_{\mathcal{D}},\label{eq:linear-L2}\\
\|u^{\mathrm{lin}}(t)\|_{\dot{H}^\alpha}\lesssim&(1+t)^{-\frac{\beta+\alpha}{2}}\|(u_0,u_1)\|_{\mathcal{D}}.\label{eq:linear-Halpha}
\end{align}
For $1<q\leq2$, $f\in L^q\cap L^2$ and $\tau\geq0$, we first prove
\begin{align}
\|K_1(\tau)*f\|_{L^2}\lesssim&(1+\tau)^{-\frac{n}{2}(\frac{1}{q}-\frac{1}{2})}\|f\|_{L^q\cap L^2},\label{eq:K1-Lq-L2-direct}\\
\|K_1(\tau)*f\|_{\dot{H}^\alpha}\lesssim&(1+\tau)^{-\frac{n}{2}(\frac{1}{q}-\frac{1}{2})-\frac{\alpha}{2}}\|f\|_{L^q\cap L^2}.\label{eq:K1-Lq-Halpha-direct}
\end{align}
For $0\leq\tau\leq1$ and $0\leq\alpha\leq1$, the formulas for $\widehat{K}_1$ give
$$
\sup_{0\leq\tau\leq1, \xi\in\mathbb{R}^n}(1+|\xi|^\alpha)|\widehat{K}_1(\tau,\xi)|<\infty.
$$
For $0\leq\tau\leq1$, we obtain \eqref{eq:K1-Lq-L2-direct}--\eqref{eq:K1-Lq-Halpha-direct} by Plancherel's identity.

For $\tau\geq1$ and $|\xi|\leq1/4$, the characteristic roots satisfy
$$
\tau_+(\xi)=-\frac{|\xi|^2}{1/2+\sqrt{1/4-|\xi|^2}}\leq-|\xi|^2, \quad \tau_-(\xi)\leq-\frac12, \quad \tau_+(\xi)-\tau_-(\xi)\geq\frac{\sqrt3}{2}.
$$
Thus $|\widehat{K}_1(\tau,\xi)|\lesssim \mathrm{e}^{-c\tau|\xi|^2}$ on $|\xi|\leq1/4$. For $1/4\leq|\xi|\leq1$, using $|\sinh z|\leq|z|\cosh|z|$ and $|\sin z|\leq|z|$, we obtain $|\widehat{K}_1(\tau,\xi)|\lesssim\tau \mathrm{e}^{-c\tau}$. For $|\xi|\geq1$, we have $|\widehat{K}_1(\tau,\xi)|\lesssim \mathrm{e}^{-\tau/2}|\xi|^{-1}$. Hence
$$
\sup_{|\xi|\geq1/4}(1+|\xi|^\alpha)|\widehat{K}_1(\tau,\xi)|\lesssim \mathrm{e}^{-c\tau}, \quad 0\leq\alpha\leq1.
$$
Let $1/r_q=1/q-1/2$, with $r_q=\infty$ if $q=2$. We have $1/2=1/r_q+1/q'$ and
$$
\big\||\xi|^\alpha \mathrm{e}^{-c\tau|\xi|^2}\big\|_{L^{r_q}(|\xi|\leq1/4)}\lesssim\tau^{-\frac{n}{2r_q}-\frac\alpha2}.
$$
By Plancherel's identity, H\"older's inequality and the Hausdorff--Young inequality, we obtain
$$
\big\||\xi|^\alpha\widehat{K}_1(\tau,\xi)\widehat{f}(\xi)\big\|_{L^2(|\xi|\leq1/4)}\lesssim\big\||\xi|^\alpha \mathrm{e}^{-c\tau|\xi|^2}\big\|_{L^{r_q}}\|\widehat{f}\|_{L^{q'}} \lesssim\tau^{-\frac{n}{2}(\frac1q-\frac12)-\frac\alpha2}\|f\|_{L^q}.
$$
By Plancherel's identity and the bound on $|\xi|\geq1/4$, we also get
$$
\big\||\xi|^\alpha\widehat{K}_1(\tau,\xi)\widehat{f}(\xi)\big\|_{L^2(|\xi|\geq1/4)}\lesssim \mathrm{e}^{-c\tau}\|f\|_{L^2}.
$$
For $\tau\geq1$, we obtain \eqref{eq:K1-Lq-Halpha-direct}. Taking $\alpha=0$, we also obtain \eqref{eq:K1-Lq-L2-direct}.

We now estimate $F(u)=\mathcal{I}_\gamma(|u|^{p_1})|u|^{p_2}$. Put
$$
m_\beta:=\frac{2n}{n+2\beta}, \quad 2^*:=\begin{cases}\infty,&n\leq2\alpha,\\ \dfrac{2n}{n-2\alpha},&n>2\alpha,\end{cases} \quad p:=p_1+p_2,
$$
where $1/2^*=0$ if $n\leq2\alpha$. For $1<q\leq2$, let
$$
c_q:=\frac1q+\frac\gamma n, \quad L_q:=\max\left\{\frac\gamma n,\frac{p_1}{2^*},c_q-1,c_q-\frac{p_2}{2}\right\}, \quad U_q:=\min\left\{1,\frac{p_1}{2},c_q-\frac{p_2}{2^*}\right\}.
$$
We choose $\theta_q$ satisfying
\begin{equation}\label{eq:theta-conditions}
\frac\gamma n<\theta_q<1, \quad \frac{p_1}{2^*}\leq\theta_q\leq\frac{p_1}{2}, \quad c_q-1<\theta_q<c_q, \quad c_q-\frac{p_2}{2}\leq\theta_q\leq c_q-\frac{p_2}{2^*}.
\end{equation}
For $q=m_\beta$ and $q=2$, we have $1<m_\beta\leq2$ and
$$
c_2\leq c_q\leq c_{m_\beta}, \quad c_{m_\beta}-1<\frac\gamma n.
$$
The assumptions of Theorem \ref{theorem:1.3} give
$$
\frac\gamma n<\frac{p_1}{2}, \quad c_{m_\beta}-1<\frac{p_2}{2}, \quad c_{m_\beta}<\frac p2,
$$
and, when $n>2\alpha$,
$$
\frac{p_1}{2^*}<1, \quad \frac{p_2}{2^*}<\frac12\leq\frac1q, \quad \frac p{2^*}\leq c_2\leq c_q.
$$
Thus, we obtain
$$
\begin{aligned}
\frac\gamma n<&\min\left\{1,\frac{p_1}{2},c_q-\frac{p_2}{2^*}\right\}, \quad \frac{p_1}{2^*}<\min\left\{1,\frac{p_1}{2}\right\}, \quad \frac{p_1}{2^*}\leq c_q-\frac{p_2}{2^*},\\
c_q-1<&\min\left\{1,\frac{p_1}{2},c_q-\frac{p_2}{2^*}\right\}, \quad c_q-\frac{p_2}{2}<\min\left\{1,\frac{p_1}{2},c_q-\frac{p_2}{2^*}\right\}.
\end{aligned}
$$
We obtain $L_q\leq U_q$. If $L_q=U_q$, then $q=2$, $n>2\alpha$ and $p/2^*=c_2$. In this case, we choose
$$
\theta_2=\frac{p_1}{2^*}=c_2-\frac{p_2}{2^*}.
$$
Since $p_1<2^*$, $p_2<n/(n-2\alpha)$ and $c_2-1<\gamma/n$, we have
$$
\max\left\{\frac\gamma n,c_2-1\right\}<\theta_2<\min\{1,c_2\}.
$$
Hence \eqref{eq:theta-conditions} also holds in this case.

In the critical case of \eqref{eq:time-threshold}, we also take $q_1<m_\beta$. Put $c_\beta=c_{m_\beta}$ and choose $\beta_1$ satisfying
$$
0<\frac{\beta_1-\beta}{n}<\min\left\{\frac12-\frac\beta n,\frac{p_2}{2}-(c_\beta-1),\frac p2-c_\beta\right\}.
$$
We set $q_1=2n/(n+2\beta_1)$. Then $\beta<\beta_1<n/2$, $1<q_1<m_\beta$ and
$$
c_{q_1}=c_\beta+\frac{\beta_1-\beta}{n}, \quad c_{q_1}-1<\frac\gamma n<\frac{p_1}{2}, \quad c_{q_1}-1<\frac{p_2}{2}, \quad c_{q_1}<\frac p2, \quad \frac p{2^*}<c_{q_1}.
$$
Since $p_1/2^*<1$ and $p_2/2^*<1/q_1$, we also obtain $L_{q_1}<U_{q_1}$. Hence we may choose $\theta_{q_1}$ satisfying \eqref{eq:theta-conditions}.

For $q=m_\beta$, $q=2$ and, in the critical case, $q=q_1$, we put
\begin{equation}\label{eq:chosen-exponents}
a_q:=\frac{p_1}{\theta_q}, \quad b_q:=\frac{p_2}{c_q-\theta_q}, \quad \frac{1}{r_{q,1}}:=\theta_q-\frac\gamma n, \quad \frac{1}{r_{q,2}}:=c_q-\theta_q.
\end{equation}
By \eqref{eq:theta-conditions} and \eqref{eq:chosen-exponents}, we have
$$
2\leq a_q,b_q\leq2^*, \quad 1<\frac{a_q}{p_1}<\frac n\gamma, \quad 1<r_{q,1},r_{q,2}<\infty, \quad \frac{1}{r_{q,1}}=\frac{p_1}{a_q}-\frac\gamma n, \quad \frac{1}{r_{q,2}}=\frac{p_2}{b_q}, \quad \frac{1}{r_{q,1}}+\frac{1}{r_{q,2}}=\frac1q.
$$
Applying Proposition \ref{Hardy-Littlewood-Sobolev} and H\"older's inequality, we obtain
$$
\|F(u(t))\|_{L^q}\lesssim\|\mathcal{I}_\gamma(|u(t)|^{p_1})\|_{L^{r_{q,1}}}\|u(t)\|_{L^{b_q}}^{p_2}\lesssim\|u(t)\|_{L^{a_q}}^{p_1}\|u(t)\|_{L^{b_q}}^{p_2}.
$$
For $2\leq a\leq2^*$, with $a<\infty$ if $n\leq2\alpha$, we put $\vartheta_a=\frac n\alpha(\frac12-\frac1a)$. By Proposition \ref{Gagliardo}, we obtain
$$
\|u(t)\|_{L^a}\lesssim\|u(t)\|_{L^2}^{1-\vartheta_a}\|u(t)\|_{\dot{H}^\alpha}^{\vartheta_a}\lesssim(1+t)^{-\frac\beta2-\frac{\alpha\vartheta_a}{2}}\|u\|_{X(T)}=(1+t)^{-\frac{n+2\beta}{4}+\frac n{2a}}\|u\|_{X(T)}.
$$
Using $p_1/a_q+p_2/b_q=c_q$, we deduce
$$
\|F(u(t))\|_{L^q}\lesssim(1+t)^{-\mu_q}\|u\|_{X(T)}^p, \quad \mu_q:=\frac{p(n+2\beta)}{4}-\frac{n}{2q}-\frac\gamma2.
$$
Since $\mu_2-\mu_q=\frac n2(1/q-1/2)\geq0$, we obtain
\begin{equation}\label{eq:nonlinear-intersection}
\|F(u(t))\|_{L^q\cap L^2}\lesssim(1+t)^{-\mu_q}\|u\|_{X(T)}^p.
\end{equation}

We also need an estimate for $F(u)-F(v)$. For $p_1,p_2\geq1$, we have
$$
F(u)-F(v)=\mathcal{I}_\gamma(|u|^{p_1})(|u|^{p_2}-|v|^{p_2})+\mathcal{I}_\gamma(|u|^{p_1}-|v|^{p_1})|v|^{p_2}.
$$
By $\big||z|^a-|w|^a\big|\lesssim|z-w|(|z|^{a-1}+|w|^{a-1})$ for $a\geq1$, Proposition \ref{Hardy-Littlewood-Sobolev} and H\"older's inequality, we get
\begin{equation}\label{eq:F-difference-Lq}
\|F(u)-F(v)\|_{L^q}\lesssim\|u\|_{L^{a_q}}^{p_1}\|u-v\|_{L^{b_q}}\left(\|u\|_{L^{b_q}}^{p_2-1}+\|v\|_{L^{b_q}}^{p_2-1}\right)+\|u-v\|_{L^{a_q}}\left(\|u\|_{L^{a_q}}^{p_1-1}+\|v\|_{L^{a_q}}^{p_1-1}\right)\|v\|_{L^{b_q}}^{p_2}.
\end{equation}
For $\|w\|_{H^\alpha}+\|z\|_{H^\alpha}\leq M$, applying \eqref{eq:F-difference-Lq} with $q=2$ and using $H^\alpha\hookrightarrow L^{a_2}\cap L^{b_2}$, we obtain
$$
\|F(w)-F(z)\|_{L^2}\leq C_M\|w-z\|_{H^\alpha}.
$$
Hence $F:H^\alpha\to L^2$ is locally Lipschitz. Using
$$
A^{p_1}B^{p_2-1}+A^{p_1-1}B^{p_2} \lesssim A^{p-1}+B^{p-1}, \quad A,B\geq0,
$$
and applying Proposition \ref{Gagliardo} to $u$, $v$ and $u-v$, we obtain, since $p_1/a_q+p_2/b_q=c_q$,
\begin{equation}\label{eq:nonlinear-difference-bound}
\|F(u(t))-F(v(t))\|_{L^q\cap L^2}\lesssim(1+t)^{-\mu_q}\|u-v\|_{X(T)}\left(\|u\|_{X(T)}^{p-1}+\|v\|_{X(T)}^{p-1}\right).
\end{equation}
We next show that the Duhamel term belongs to $\mathcal{C}([0,T],H^\alpha)$. For every $f\in L^2$ and $\tau_0\in[0,T]$, the formulas for $\widehat{K}_1$ give
$$
\sup_{0\leq\tau\leq T, \xi\in\mathbb{R}^n}(1+|\xi|^\alpha)|\widehat{K}_1(\tau,\xi)|<\infty, \quad \widehat{K}_1(\tau,\xi)\longrightarrow\widehat{K}_1(\tau_0,\xi) \quad (\tau\to\tau_0).
$$
By Plancherel's identity and dominated convergence, we obtain
$$
\|(K_1(\tau)-K_1(\tau_0))*f\|_{H^\alpha}\longrightarrow0 \quad (\tau\to\tau_0).
$$
For $u\in X(T)$, the local Lipschitz estimate gives $F(u)\in\mathcal{C}([0,T],L^2)$. Hence, for $0\leq s,s_0\leq t\leq T$, we have
$$
\begin{aligned}
&\|K_1(t-s)*F(u(s))-K_1(t-s_0)*F(u(s_0))\|_{H^\alpha}\\
\lesssim&\|F(u(s))-F(u(s_0))\|_{L^2}+\|(K_1(t-s)-K_1(t-s_0))*F(u(s_0))\|_{H^\alpha}\longrightarrow0 \quad \text{ as } s\to s_0.
\end{aligned}
$$
We may therefore define
$$
u^{\mathrm{non}}(t):=\int_0^tK_1(t-s)*F(u(s))ds \quad \text{ in } H^\alpha.
$$
For $h>0$ and $t+h\leq T$, we have
$$
u^{\mathrm{non}}(t+h)-u^{\mathrm{non}}(t)=\int_0^t\big(K_1(t+h-s)-K_1(t-s)\big)*F(u(s))ds+\int_t^{t+h}K_1(t+h-s)*F(u(s))ds.
$$
For $0\leq s\leq t$, we have
$$
\|\big(K_1(t+h-s)-K_1(t-s)\big)*F(u(s))\|_{H^\alpha}\lesssim C_T\|F(u(s))\|_{L^2}.
$$
By dominated convergence and \eqref{eq:K1-Lq-Halpha-direct} with $q=2$, we obtain
$$
\|u^{\mathrm{non}}(t+h)-u^{\mathrm{non}}(t)\|_{H^\alpha} \lesssim \left\|\int_0^t\big(K_1(t+h-s)-K_1(t-s)\big)*F(u(s))ds\right\|_{H^\alpha}+\int_t^{t+h}\|F(u(s))\|_{L^2}ds\longrightarrow0
$$
as $h\to0^+$. For $h<0$, the same argument gives $u^{\mathrm{non}}(t+h)\to u^{\mathrm{non}}(t)$ in $H^\alpha$. Hence $u^{\mathrm{non}}\in\mathcal{C}([0,T],H^\alpha)$.

We estimate $u^{\mathrm{non}}$. For $\ell=0,\alpha$, \eqref{eq:K1-Lq-L2-direct}--\eqref{eq:K1-Lq-Halpha-direct} and \eqref{eq:nonlinear-intersection} give
$$
\begin{aligned}
\int_0^{t/2}\||D|^\ell(K_1(t-s)*F(u(s)))\|_{L^2}ds \lesssim& (1+t)^{-\frac n2(\frac1q-\frac12)-\frac\ell2}\|u\|_{X(T)}^p\int_0^{t/2}(1+s)^{-\mu_q}ds,\\
\int_{t/2}^t\||D|^\ell(K_1(t-s)*F(u(s)))\|_{L^2}ds \lesssim& (1+t)^{-\mu_2}\|u\|_{X(T)}^p\int_0^{t/2}(1+\tau)^{-\frac\ell2}d\tau,
\end{aligned}
$$
where $\tau=t-s$ in the second integral.

If $p>1+(4+2\gamma)/(n+2\beta)$, we take $q=m_\beta$ in the first integral. We have
$$
\mu_{m_\beta}=\frac{(p-1)(n+2\beta)}4-\frac\gamma2>1, \quad \mu_2=\mu_{m_\beta}+\frac\beta2>1+\frac\beta2.
$$
Using $n(1/m_\beta-1/2)=\beta$ and $\int_0^{t/2}(1+\tau)^{-\ell/2}d\tau\lesssim(1+t)^{1-\ell/2}$, we obtain
$$
\int_0^t\||D|^\ell(K_1(t-s)*F(u(s)))\|_{L^2}ds\lesssim(1+t)^{-\frac{\beta+\ell}{2}}\|u\|_{X(T)}^p, \quad \ell=0,\alpha.
$$
If $p=1+(4+2\gamma)/(n+2\beta)$, we take $q=q_1$ in the first integral. With $\delta=\beta_1-\beta>0$, we have
$$
\mu_{q_1}=1-\frac\delta2, \quad \mu_2=1+\frac\beta2, \quad \frac n2\left(\frac1{q_1}-\frac12\right)=\frac{\beta_1}{2}, \quad \int_0^{t/2}(1+s)^{-\mu_{q_1}}ds\lesssim(1+t)^{\delta/2}.
$$
Therefore, for $\ell=0,\alpha$, we obtain
$$
\begin{aligned}
\int_0^{t/2}\||D|^\ell(K_1(t-s)*F(u(s)))\|_{L^2}ds\lesssim&(1+t)^{-\frac{\beta_1+\ell}{2}+\frac\delta2}\|u\|_{X(T)}^p=(1+t)^{-\frac{\beta+\ell}{2}}\|u\|_{X(T)}^p,\\
\int_{t/2}^t\||D|^\ell(K_1(t-s)*F(u(s)))\|_{L^2}ds\lesssim&(1+t)^{-1-\frac\beta2+1-\frac\ell2}\|u\|_{X(T)}^p=(1+t)^{-\frac{\beta+\ell}{2}}\|u\|_{X(T)}^p.
\end{aligned}
$$
In both cases, we have
\begin{equation}\label{eq:nonlinear-decay}
\begin{cases}
\|u^{\mathrm{non}}(t)\|_{L^2}
\lesssim (1+t)^{-\frac\beta2}\|u\|_{X(T)}^p,\\
\|u^{\mathrm{non}}(t)\|_{\dot{H}^\alpha}
\lesssim (1+t)^{-\frac{\beta+\alpha}{2}}\|u\|_{X(T)}^p.
\end{cases}
\end{equation}
By \eqref{eq:linear-L2}, \eqref{eq:linear-Halpha} and
\eqref{eq:nonlinear-decay}, we obtain \eqref{ineq:5.1}
with a constant independent of $T$.

We now prove \eqref{ineq:5.2}. Put
$$
D_T(u,v):=\|u-v\|_{X(T)}\left(\|u\|_{X(T)}^{p-1}+\|v\|_{X(T)}^{p-1}\right), \quad G(s):=F(u(s))-F(v(s)).
$$
For $\ell=0,\alpha$, we apply \eqref{eq:K1-Lq-L2-direct}--\eqref{eq:K1-Lq-Halpha-direct} and \eqref{eq:nonlinear-difference-bound}. On $[0,t/2]$, we take $q=m_\beta$ if $p>1+(4+2\gamma)/(n+2\beta)$, and $q=q_1$ in the critical case. On $[t/2,t]$, we take $q=2$ in both cases. Since $\mu_{m_\beta}>1$, $\mu_{q_1}=1-(\beta_1-\beta)/2$ and $\mu_2\geq1+\beta/2$, we obtain
\begin{equation}\label{eq:Duhamel-difference-decay}
\int_0^t\||D|^\ell(K_1(t-s)*G(s))\|_{L^2}ds \lesssim(1+t)^{-\frac{\beta+\ell}{2}}D_T(u,v), \quad \ell=0,\alpha.
\end{equation}
For $\ell=0,\alpha$, multiplying \eqref{eq:Duhamel-difference-decay} by $(1+t)^{\frac{\beta+\ell}{2}}$, we obtain
$$
\sup_{0\leq t\leq T}(1+t)^{\frac{\beta+\ell}{2}} \left\|\int_0^t|D|^\ell(K_1(t-s)*G(s))ds\right\|_{L^2} \lesssim D_T(u,v).
$$
By the definition of $X(T)$, we obtain \eqref{ineq:5.2} with a constant independent of $T$.

We choose $R_*>0$ such that
$$
C_1(2R_*)^p\leq R_*, \quad 2C_2(2R_*)^{p-1}<1,
$$
and take $\varepsilon>0$ with $C_1\varepsilon\leq R_*$. We set
$$
\mathcal{B}_{2R_*}(T):=\{u\in X(T):\|u\|_{X(T)}\leq2R_*\}.
$$
For $\|(u_0,u_1)\|_{\mathcal{D}}\leq\varepsilon$ and $u,v\in\mathcal{B}_{2R_*}(T)$, \eqref{ineq:5.1}--\eqref{ineq:5.2} give
$$
\|Nu\|_{X(T)}\leq C_1\varepsilon+C_1(2R_*)^p\leq2R_*, \quad \|Nu-Nv\|_{X(T)}\leq 2C_2(2R_*)^{p-1}\|u-v\|_{X(T)}.
$$
For every $T>0$, we obtain a unique fixed point in this ball by the Banach fixed-point theorem. For $0<T_1<T_2$, uniqueness on $[0,T_1]$ gives $u_{T_2}|_{[0,T_1]}=u_{T_1}$. Thus we obtain a global mild solution $u\in\mathcal{C}([0,\infty),H^\alpha)$.

We prove uniqueness in $\mathcal{C}([0,\infty),H^\alpha)$. Let $u,v$ be two mild solutions with the same initial data. For $T>0$, put $M=\sup_{0\leq t\leq T}(\|u(t)\|_{H^\alpha}+\|v(t)\|_{H^\alpha})<\infty$. Since $\|F(u(t))-F(v(t))\|_{L^2}\leq C_M\|u(t)-v(t)\|_{H^\alpha}$, applying \eqref{eq:K1-Lq-L2-direct}--\eqref{eq:K1-Lq-Halpha-direct} with $q=2$, we have
$$
\|u(t)-v(t)\|_{H^\alpha}\leq\int_0^t\|K_1(t-s)*(F(u(s))-F(v(s)))\|_{H^\alpha}ds \lesssim C_M\int_0^t\|u(s)-v(s)\|_{H^\alpha}ds, \quad 0\leq t\leq T.
$$
Applying Gronwall's inequality, we get $u=v$ on $[0,T]$. Since $T>0$ is arbitrary, the global mild solution is unique.

Finally, \eqref{ineq:5.1} and $C_1(2R_*)^{p-1}\leq1/2$ give
$$
\|u\|_{X(T)}\leq C_1\|(u_0,u_1)\|_{\mathcal{D}}+C_1\|u\|_{X(T)}^p \leq C_1\|(u_0,u_1)\|_{\mathcal{D}}+\frac12\|u\|_{X(T)}.
$$
Hence $\|u\|_{X(T)}\lesssim\|(u_0,u_1)\|_{\mathcal{D}}$ uniformly in $T$. By the definition of $X(T)$, we obtain \eqref{ineq:theorem:1.3.1} and \eqref{ineq:theorem:1.3.2}. This completes the proof.
\end{proof}

\section{Nonexistence of global weak solutions} \label{section_6}
In this section, we prove Theorem \ref{theorem:1.5}. We begin with the definition of weak solutions to \eqref{eq:1.0}.
\begin{definition}[Weak solution]\label{def:weak}
Let $T\in(0,\infty]$, $p_1,p_2>0$ with $p_1+p_2>2$, $\gamma\in(0,n)$ and $u_0,u_1\in L_\mathrm{loc}^1\left(\mathbb{R}^n\right)$. Put $\widetilde p=(p_1+p_2)/2>1$. We define $\mathcal{I}_\gamma(|u|^{p_1})$ by \eqref{eq:Riesz} as a nonnegative extended integral and set $\mathcal{I}_\gamma(|u|^{p_1})|u|^{p_2}=0$ on $\{u=0\}$.

We call $u\in L_\mathrm{loc}^{\widetilde p}([0,T)\times\mathbb{R}^n)$ a weak solution of \eqref{eq:1.0} on $[0,T)\times\mathbb{R}^n$ if
$$
\mathcal{I}_\gamma(|u|^{p_1})|u|^{p_2} \in L_\mathrm{loc}^1([0,T)\times \mathbb{R}^n)
$$
and
\begin{align}\label{eq:8.1}
&\int_0^T\int_{\mathbb{R}^n}\mathcal{I}_\gamma(|u|^{p_1})|u|^{p_2}\phi(t,x)dxdt+\int_{\mathbb{R}^n}(u_0+u_1)\phi(0,x)dx-\int_{\mathbb{R}^n} u_0\partial_t\phi(0,x)dx \notag\\
=&\int_0^T\int_{\mathbb{R}^n}u(\partial_t^2\phi-\Delta\phi-\partial_t\phi)dxdt
\end{align}
for every $\phi\in\mathcal{C}_0^\infty([0,T)\times\mathbb{R}^n)$. If $T=\infty$, we call $u$ a global weak solution.
\end{definition}
We are now in a position to prove Theorem \ref{theorem:1.5}.
\begin{proof}[Proof of Theorem \ref{theorem:1.5}]
Assume that $u=u(t,x)$ is a global weak solution to \eqref{eq:1.0}. Put $p=p_1+p_2$, $\widetilde p=p/2>1$, and let $\widetilde q=\widetilde p/(\widetilde p-1)$.

We choose a nonincreasing function $\chi_0\in\mathcal{C}_0^\infty([0,\infty))$ and a radial function $\varphi_0\in\mathcal{C}_0^\infty\left(\mathbb{R}^n\right)$ such that $0\leq\chi_0,\varphi_0\leq1$, $\chi_0=1$ on $[0,1/2]$, $\chi_0=0$ on $[1,\infty)$, $\varphi_0=1$ on $\{|x|\leq1/2\}$ and $\varphi_0=0$ on $\{|x|\geq1\}$. Let $k\geq2\widetilde q$ be an integer and put $\chi=\chi_0^k$, $\varphi=\varphi_0^k$. Since $\widetilde q/\widetilde p=\widetilde q-1$, we have
\begin{equation}\label{eq:test-cutoff-bounds}
\begin{cases}
\chi^{-\frac{\widetilde q}{\widetilde p}}|\chi'|^{\widetilde q}\lesssim \chi_0^{k-\widetilde q}\lesssim1,\\
\chi^{-\frac{\widetilde q}{\widetilde p}}|\chi''|^{\widetilde q}\lesssim \chi_0^{k-2\widetilde q}+\chi_0^{k-\widetilde q}\lesssim1,\\
\varphi^{-\frac{\widetilde q}{\widetilde p}}|\Delta\varphi|^{\widetilde q}\lesssim \varphi_0^{k-2\widetilde q}+\varphi_0^{k-\widetilde q}\lesssim1.
\end{cases}
\end{equation}
Here we take these expressions to be zero wherever the cutoffs vanish.

For $R>1$, we put $\mathcal{B}_R=\{x\in\mathbb{R}^n:|x|\leq R\}$ and
$$
\chi_R(t)=\chi(R^{-2}t), \quad \varphi_R(x)=\varphi(R^{-1}x), \quad \phi_R(t,x)=\chi_R(t)\varphi_R(x).
$$
We also put
$$
\mathcal{K}_R:=\int_0^{R^2}\int_{\mathcal{B}_R}\mathcal{I}_\gamma(|u|^{p_1})|u|^{p_2}\phi_R dxdt, \quad \mathcal{M}_R:=\int_0^{R^2}\int_{\mathcal{B}_R}|u|^{\widetilde p}\phi_R dxdt.
$$
By Definition \ref{def:weak}, we have $\mathcal{K}_R,\mathcal{M}_R<\infty$.

Since $\phi_R(0,x)=\varphi_R(x)$ and $\partial_t\phi_R(0,x)=0$, applying \eqref{eq:8.1} with $\phi=\phi_R$, we obtain
\begin{equation}\label{ineq:testfunction}
\mathcal{K}_R+\int_{\mathcal{B}_R}(u_0+u_1)\varphi_R dx=\int_0^{R^2}\int_{\mathcal{B}_R}u(\partial_t^2\chi_R-\partial_t\chi_R)\varphi_R dxdt-\int_0^{R^2}\int_{\mathcal{B}_R}u\chi_R\Delta\varphi_R dxdt.
\end{equation}

By \eqref{eq:test-cutoff-bounds} and the changes of variables $t=R^2\tau$, $x=Ry$, we have
$$
\begin{aligned}
\int_0^{R^2}\int_{\mathcal{B}_R}\chi_R^{-\frac{\widetilde q}{\widetilde p}}|\partial_t^2\chi_R|^{\widetilde q}\varphi_R dxdt\lesssim& R^{n+2-4\widetilde q},\\
\int_0^{R^2}\int_{\mathcal{B}_R}\chi_R\varphi_R^{-\frac{\widetilde q}{\widetilde p}}|\Delta\varphi_R|^{\widetilde q}dxdt\lesssim& R^{n+2-2\widetilde q},\\
\int_0^{R^2}\int_{\mathcal{B}_R}\chi_R^{-\frac{\widetilde q}{\widetilde p}}|\partial_t\chi_R|^{\widetilde q}\varphi_R dxdt\lesssim& R^{n+2-2\widetilde q}.
\end{aligned}
$$
Applying H\"older's inequality over
$(0,R^2)\times\mathcal{B}_R$, we obtain
\begin{align}
\left|\int_0^{R^2}\int_{\mathcal{B}_R}u\partial_t^2\chi_R\varphi_R dxdt\right|\lesssim&\mathcal{M}_R^{1/\widetilde p}R^{-4+\frac{n+2}{\widetilde q}},\label{ineq:J_1}\\
\left|\int_0^{R^2}\int_{\mathcal{B}_R}u\chi_R\Delta\varphi_R dxdt\right|+\left|\int_0^{R^2}\int_{\mathcal{B}_R}u\partial_t\chi_R\varphi_R dxdt\right|\lesssim&\mathcal{M}_R^{1/\widetilde p}R^{-2+\frac{n+2}{\widetilde q}}.\label{ineq:J_3}
\end{align}

We next establish a lower bound for $\mathcal{K}_R$. For $x,y\in\mathcal{B}_R$, we have $|x-y|\leq2R$. By \eqref{eq:Riesz}, we obtain
$$
\mathcal{I}_\gamma(|u(t,\cdot)|^{p_1})(x)\gtrsim R^{-(n-\gamma)}\int_{\mathcal{B}_R}|u(t,y)|^{p_1}dy, \quad x\in\mathcal{B}_R.
$$
For almost every $t\in(0,R^2)$, we set
$$
\mathcal{H}_R(t)=\int_{\mathcal{B}_R}|u(t,x)|^{\widetilde p}\phi_R(t,x)dx, \quad A_i(t)=\int_{\mathcal{B}_R}|u(t,x)|^{p_i}\phi_R(t,x)dx, \quad i=1,2.
$$
Since $2\widetilde p=p_1+p_2$ and $0\leq\phi_R\leq1$, applying Cauchy--Schwarz to $\min\{|u|,N\}$ and then letting $N\to\infty$ by monotone convergence, we obtain
$$
\mathcal{H}_R(t)^2\leq A_1(t)A_2(t) \leq\left(\int_{\mathcal{B}_R}|u(t,y)|^{p_1}dy\right)A_2(t) \lesssim R^{n-\gamma}\int_{\mathcal{B}_R}\mathcal{I}_\gamma(|u(t,\cdot)|^{p_1})(x)|u(t,x)|^{p_2}\phi_R(t,x)dx.
$$
Integrating over $(0,R^2)$, we obtain
\begin{equation}\label{ineq:44}
\mathcal{K}_R\gtrsim R^{-(n-\gamma)}\int_0^{R^2}\mathcal{H}_R(t)^2dt.
\end{equation}
Applying Cauchy--Schwarz's inequality, we obtain
\begin{equation}\label{ineq:45}
\mathcal{M}_R^{1/\widetilde p}=\left(\int_0^{R^2}\mathcal{H}_R(t)dt\right)^{1/\widetilde p}\leq R^{1/\widetilde p} \left(\int_0^{R^2}\mathcal{H}_R(t)^2dt\right)^{1/(2\widetilde p)}.
\end{equation}

For $R\geq4$, we have $\varphi_R=1$ on $\{|x|\leq R/2\}$. Applying \eqref{condition_u0u1} on $R/4\leq|x|\leq R/2$, we obtain
\begin{equation}\label{ineq:8.10}
\int_{\mathcal{B}_R}(u_0+u_1)\varphi_R dx\gtrsim R^{\frac{n}{2}-\beta}(\log R)^{-1}.
\end{equation}

Put $A_R=\int_0^{R^2}\mathcal{H}_R(t)^2dt<\infty$. By \eqref{ineq:testfunction}, \eqref{ineq:J_1}, \eqref{ineq:J_3}, \eqref{ineq:44}, \eqref{ineq:45} and \eqref{ineq:8.10}, we obtain
$$
R^{-(n-\gamma)}A_R+cR^{\frac{n}{2}-\beta}(\log R)^{-1}\lesssim A_R^{1/(2\widetilde p)}R^{-2+\frac{n+2}{\widetilde q}+\frac{1}{\widetilde p}}, \quad R\geq4,
$$
where $c>0$ is independent of $R$. Multiplying by $R^{n-\gamma}$ and applying Young's inequality, we get
$$
A_R+cR^{\frac{3n}{2}-\beta-\gamma}(\log R)^{-1} \leq C A_R^{1/(2\widetilde p)}R^{-2+n-\gamma+\frac{n+2}{\widetilde q}+\frac{1}{\widetilde p}} \leq \frac12 A_R+CR^{\frac{2\widetilde p}{2\widetilde p-1}\left(-2+n-\gamma+\frac{n+2}{\widetilde q}+\frac{1}{\widetilde p}\right)}.
$$
Since $1/\widetilde p=2/p$ and $1/\widetilde q=1-2/p$, we have
$$
-2+n-\gamma+\frac{n+2}{\widetilde q}+\frac{1}{\widetilde p}=2n-\gamma-\frac{2n+2}{p}.
$$
Therefore
$$
R^{\frac{3n}{2}-\beta-\gamma}(\log R)^{-1}\lesssim R^{\frac{p}{p-1}\left(2n-\gamma-\frac{2n+2}{p}\right)}.
$$
By \eqref{p_critical}, we have
$$
\delta:=\frac{3n}{2}-\beta-\gamma-\frac{p}{p-1}\left(2n-\gamma-\frac{2n+2}{p}\right)=\frac{n+2\beta+4+2\gamma-p(n+2\beta)}{2(p-1)}>0.
$$
Thus $1\lesssim R^{-\delta}\log R\to0$ as $R\to\infty$, which is a contradiction. This completes the proof.
\end{proof}


\section*{Acknowledgments}
Duc An Phan sincerely acknowledges the financial support provided by the Banking Academy of Vietnam.

\appendix

\section{Auxiliary inequalities}\label{section_tools}

\begin{lemma}[Bernstein's inequality, \cite{Bahouri2011}]\label{lemma:bernstein}
Let $1\leq a\leq b<\infty$. For $j\in\mathbb{Z}$, assume that $\dot{\Delta}_j f\in L^a$. Then
$$
\|\dot{\Delta}_j f\|_{L^b}\lesssim2^{jn(1/a-1/b)}\|\dot{\Delta}_j f\|_{L^a}.
$$
In particular, let $f\in L^a$ and put $\delta=n(1/a-1/b)$. Then
$$
\|f\|_{\dot{B}_{b,\infty}^{-\delta}}\lesssim\|f\|_{L^a}.
$$
\end{lemma}

\begin{proposition}[Fractional Gagliardo-Nirenberg inequality, \cite{Molinet2011}] \label{Gagliardo}
Let $0<\alpha\leq1$. If $n\leq2\alpha$, let $2\leq r<\infty$, while if $n>2\alpha$, let $2\leq r\leq\frac{2n}{n-2\alpha}$. Then
$$
\left\|u\right\|_{L^r}\lesssim \left\|u\right\|_{L^2}^{1-\omega}\left\|u\right\|_{\dot{H}^\alpha}^{\omega},
\quad
\omega:=\frac{n}{\alpha}\left(\frac{1}{2}-\frac{1}{r}\right)\in[0,1].
$$
\end{proposition}
\begin{proposition}[Hardy-Littlewood-Sobolev inequality, \cite{Lieb1983}] \label{Hardy-Littlewood-Sobolev}
Let $0<\gamma<n$ and $1<m_2<m_1<\infty$ such that $\frac{1}{m_1}=\frac{1}{m_2}-\frac{\gamma}{n}$. Then there exists a constant $C=C(n,\gamma,m_2)>0$ such that
$$
\|\mathcal{I}_\gamma(f)\|_{L^{m_1}}\leq C\|f\|_{L^{m_2}}.
$$
\end{proposition}

\end{document}